\documentclass[12pt]{amsart}
\usepackage{amsmath}
\usepackage{amstext}
\usepackage{amsfonts}
\usepackage{amssymb}
\usepackage{amsthm}
\usepackage{amsrefs}
\usepackage{bm}
\usepackage{csquotes}

\usepackage{microtype}
\usepackage{color,hyperref}
\definecolor{darkblue}{rgb}{0.0,0.0,0.3}
\hypersetup{colorlinks,breaklinks,
            linkcolor=red,urlcolor=red,
            anchorcolor=red,citecolor=red}

\usepackage{tikz-cd}
\usetikzlibrary{arrows}
\usetikzlibrary{matrix,arrows,decorations.pathmorphing}

\theoremstyle{plain}
\newtheorem{thm}{Theorem}[section]

\newtheorem{cor}[thm]{Corollary}
\newtheorem{prop}[thm]{Proposition}
\newtheorem{lem}[thm]{Lemma}

\theoremstyle{definition}
\newtheorem{defn}[thm]{Definition}
\newtheorem{example}[thm]{Example}
\newtheorem{rem}[thm]{Remark}

\numberwithin{equation}{section}

\newcommand{\B}{{\mathcal{B}}}
\newcommand{\K}{{\mathcal{K}}}
\renewcommand{\P}{{\mathcal{P}}}

\newcommand{\bC}{{\mathbb{C}}}

\newcommand{\bK}{{\mathbf{K}}}
\newcommand{\bL}{{\mathbf{L}}}
\newcommand{\bN}{{\mathbf{N}}}
\newcommand{\bR}{{\mathbf{R}}}

\newcommand{\fb}{\partial_F G}

\newcommand{\Ad}{\operatorname{Ad}}
\newcommand{\Aut}{\operatorname{Aut}}
\newcommand{\id}{\operatorname{id}}
\newcommand{\Id}{\operatorname{Id}}
\newcommand{\pr}{\operatorname{pr}}

\newcommand{\Glimm}{\operatorname{Glimm}}
\newcommand{\Prim}{\operatorname{Prim}}
\newcommand{\Sub}{\operatorname{Sub}}

\newcommand{\hull}{\operatorname{hull}}
\newcommand{\supp}{\operatorname{supp}}

\begin{document}
\title[The ideal structure of reduced crossed products]{Noncommutative boundaries and the ideal structure of reduced crossed products}

\author[M. Kennedy]{Matthew Kennedy}
\address{Department of Pure Mathematics\\ University of Waterloo\\
Waterloo, Ontario \; N2L 3G1 \\Canada}
\email{matt.kennedy@uwaterloo.ca}

\author[C. Schafhauser]{Christopher Schafhauser}
\address{Department of Pure Mathematics\\ University of Waterloo\\
Waterloo, Ontario \; N2L 3G1 \\Canada}
\email{cschafhauser@uwaterloo.ca}

\begin{abstract}
A C*-dynamical system is said to have the ideal separation property if every ideal in the corresponding crossed product arises from an invariant ideal in the C*-algebra. In this paper we characterize this property for unital C*-dynamical systems over discrete groups. To every C*-dynamical system we associate a ``twisted'' partial C*-dynamical system that encodes much of the structure of the action. This system can often be ``untwisted,'' for example when the algebra is commutative, or when the algebra is prime and a certain specific subgroup has vanishing Mackey obstruction. In this case, we obtain relatively simple necessary and sufficient conditions for the ideal separation property. A key idea is a notion of noncommutative boundary for a C*-dynamical system that generalizes Furstenberg's notion of topological boundary for a group.
\end{abstract}

\subjclass[2010]{Primary 46L35; Secondary 45L55, 47L65, 43A65}
\keywords{C*-dynamical system, reduced crossed product, ideal structure, noncommutative boundary}
\thanks{First author supported by NSERC Grant Number 418585.}

\maketitle

\tableofcontents
\addtocontents{toc}{\setcounter{tocdepth}{2}}

\section{Introduction}

A C*-dynamical system is a triple $(A,G,\alpha)$ consisting of a unital C*-algebra $A$, a countable discrete group $G$ and an action of $G$ on $A$ in the form of a group homomorphism $\alpha: G \to \Aut(A)$. In this paper, we consider the ideal structure of the corresponding reduced crossed product $A \times_r G$.

The ideal structure of $A \times_r G$ is as nice as possible when every ideal in $A \times_r G$ arises from an invariant ideal in $A$, in the sense that the intersection map $J \to A \cap J$ is injective on the ideals of $A \times_r G$. In this case, we say that $(A,G,\alpha)$ has the {\em ideal separation property}. The main result in this paper is a characterization of this property for a large class of noncommutative C*-dynamical systems in terms of the dynamics of $(A,G,\alpha)$.

It is convenient to characterize the ideal separation property in terms of a weaker property. We say that $(A,G,\alpha)$ has the {\em intersection property} if every nonzero ideal in $A \times_r G$ has nonzero intersection with $A$. Sierakowski \cite{S2010} showed that if $(A,G,\alpha)$ is exact (which is almost always the case) and every quotient has the intersection property, then $(A,G,\alpha)$ has the ideal separation property if and only if it has the intersection property.

Kawabe \cite{K2017} recently obtained a characterization of the ideal separation property for exact commutative C*-dynamical systems in terms of the intersection property. His results generalize results obtained by the first author and his collaborators \cites{KK2017,BKKO2017,K2015} relating to the simplicity of reduced C*-algebras of groups.

Significant difficulties arise in the study of noncommutative dynamics that have no analogue in commutative dynamics, in particular because of the existence of non-trivial inner automorphisms. A key point in this paper is a new method for handling these difficulties by considering a partial C*-dynamical system arising from the injective envelope of the C*-algebra. The theory of partial C*-dynamical systems has been extensively developed by Exel and his collaborators, and we point out some connections with their work.

Using the partial C*-dynamical system arising from the injective envelope, we identify a specific obstruction to understanding the ideal structure of the reduced crossed product. And fortunately, it turns out that it is often possible to remove this obstruction by ``untwisting'' it. In this case, we say that the corresponding C*-dynamical system has {\em vanishing obstruction}.

If $A$ is prime, then $(A,G,\alpha)$ has vanishing obstruction precisely when the familiar Mackey obstruction $H^2(G_A,\mathbb{T})$ vanishes for the subgroup $G_A = \{s \in G : \alpha_s \text{ is quasi-inner} \}$. This is the subgroup of elements in $G$ that are ``almost inner'' in a certain precise sense. This happens in many cases, for example, if either $G_A$ is free or $G_A$ is finite and every Sylow subgroup is cyclic. In particular, it vanishes if $G_A$ is cyclic.

The starting point in our paper is the observation that for C*-dynamical systems with vanishing obstruction, the intersection property can be completely characterized in terms of the dynamics of the action. The precise statement of the characterization splits into two cases, depending on whether the group is amenable or not.

In the amenable case, since $(A,G,\alpha)$ has vanishing obstruction if $A$ is abelian, we obtain a significant generalization of a result of Kawamura and Tomiyama \cite{KT1990}*{Theorem 4.1} (see also \cite{AS1994}*{Theorem 2}).

\begin{thm}
Let $(A,G,\alpha)$ be a C*-dynamical system such that $G$ is amenable. If $(A,G,\alpha)$ is properly outer, then $(A,G,\alpha)$ has the intersection property. If $(A,G,\alpha)$ has vanishing obstruction, then the converse is also true.
\end{thm}

In the non-amenable case, it is necessary to consider the injective hull of a C*-dynamical system, and we obtain a significant generalization of the characterization of C*-simplicity established by Kalantar and the first author \cite{KK2017}.

\begin{thm}
Let $(A,G,\alpha)$ be a C*-dynamical system such that $G$ is non-amenable. If the injective hull $(I_G(A),G,I_G(\alpha))$ of $(A,G,\alpha)$ is properly outer, and in particular if $(A,G,\alpha)$ is properly outer, then $(A,G,\alpha)$ has the intersection property. If $(A,G,\alpha)$ has vanishing obstruction, then the converse is also true. 
\end{thm}

The notion of proper outerness for a C*-dynamical system is a noncommutative generalization of the notion of topological freeness from topological dynamics. Proper outerness was defined for von Neumann algebras by Kallmann \cite{K1969}, and for general C*-algebras by Elliott \cite{E1980}. We adopt Kishimoto's definition \cite{K1982} of proper outerness in terms of the Borchers spectrum, which is equivalent to Elliott's definition in the separable setting. We are grateful to the anonymous referees for bringing the recent paper \cite{KM2016} to our attention, which considers generalizations of these notions for Fell bundles.

A key point in our paper is a more dynamical characterization of proper outerness in terms of what we call {\em quasi-stabilizer subgroups}. Let $\Glimm(A)$ denote the space of Glimm ideals in $A$, which can be identified with the spectrum of the center of $A$. We identify a class of ideals in $A$ that we call {\em pseudo-Glimm ideals} since they behave in many ways like Glimm ideals. Let $\Glimm_p(A)$ denote the space of pseudo-Glimm ideals in $A$. It turns out that $\Glimm_p(A)$ is compact and Hausdorff in the strong topology (or Fell topology) on the set of ideals of $A$. In many cases $\Glimm_p(A)$ and $\Glimm(A)$ coincide, for example if $A$ has Hausdorff primitive ideal space. But it is possible for these spaces to differ.

For a pseudo-Glimm ideal $I \in \Glimm_p(A)$, the corresponding quasi-stabilizer subgroup $G_I$ consists of the elements in $s \in G$ such that the automorphism $\alpha_s$ is almost inner on the quotient $A/I$ in a certain precise sense. Quasi-stabilizer subgroups are a kind of noncommutative generalization of neighborhood stabilizer subgroups from topological dynamics. We show that they completely charaterize proper outerness, in the sense that $(A,G,\alpha)$ is properly outer if and only if for every pseudo-Glimm ideal $I \in \Glimm_p(A)$, the corresponding quasi-stabilizer $G_I$ is trivial.

With a great deal more effort, we obtain a more intrinsic characterization of the intersection property for C*-dynamical systems with vanishing obstruction, i.e. a characterization that does not refer to the injective hull. Before stating this result, we require some terminology.

We equip the space $\Glimm_p(A)$ of pseudo-Glimm ideals of $A$ with the strong topology and let $\operatorname{Sub}(G)$ denote the compact Hausdorff space of subgroups of $G$ equipped with the Chabauty topology. We then consider the action of $G$ on the space $\Glimm_p(A) \times \Sub(G)$  equipped with the product topology. An invariant closed subset $X \subseteq \Glimm_p(A) \times \operatorname{Sub}(G)$ is said to be {\em covering for $(A,G,\alpha)$} if
\begin{enumerate}
\item The ideals in $\pr_1(X)$ cover $X$ in a certain specific sense, where $\pr_1 : \Glimm(A) \times \Sub(G) \to \Glimm(A)$ denotes the projection onto the first coordinate, and
\item For each $(I,H) \in X$, $H \leq G_I$, where $G_I$ denotes the quasi-stabilizer subgroup corresponding to $I$.
\end{enumerate}
We say that an $(A,G,\alpha)$-covering subset $X \subseteq \Glimm_p(A) \times \Sub(G)$ is {\em amenable} if $\pr_2(X) \subseteq \Sub_a(G)$, where  $\pr_2 : \Glimm(A) \times \operatorname{Sub}(G) \to \Sub(G)$ denotes the projection onto the second coordinate and $\operatorname{Sub}_a(G)$ denotes the set of amenable subgroups of $G$. If $\pr_2(X) = \{\{e\}\}$, then we say that $X$ is trivial.

\begin{thm} \label{thm:intro-main}
Let $(A,G,\alpha)$ be a C*-dynamical system. If every amenable $(A,G,\alpha)$-covering subset of $\Glimm_p(A) \times \Sub(G)$ contains a trivial $(A,G,\alpha)$-covering subset, then $(A,G,\alpha)$ has the intersection property. If $(A,G,\alpha)$ has vanishing obstruction, then the converse is also true.
\end{thm}

If $A = \bC$, then $(A,G,\alpha)$ has vanishing obstruction and the minimal amenable $(A,G,\alpha)$-covering subsets of $\Glimm_p(A) \times \Sub(G)$ are precisely the uniformly recurrent subgroups of $G$, so the above theorem generalizes the characterization of C*-simplicity obtained by the first author \cite{K2015}*{Theorem 4.1}. More generally, if $A$ is commutative, then $(A,G,\alpha)$ has vanishing obstruction and the above theorem reduces to the characterization of the intersection property for commutative C*-dynamical systems obtained by Kawabe \cite{K2017}*{Theorem 5.2}.

If $A$ is prime, then the hypotheses of Theorem \ref{thm:intro-main} are easier to verify.

\begin{cor}
Let $(A,G,\alpha)$ be a C*-dynamical system such that $A$ is prime and let $G_A = \{s \in G : \alpha_s \text{ is quasi-inner} \}$. If every amenable $(A,G)$-covering subset of $\Glimm_p(A) \times \Sub(G_A)$ contains a trivial $(A,G)$-covering subset, then $(A,G,\alpha)$ has the intersection property. If $(A,G,\alpha)$ has vanishing Mackey obstruction $H^2(G_A,\mathbb{T})$, then the converse is also true.
\end{cor}

If $(A,G,\alpha)$ is minimal, meaning that $A$ has no non-trivial invariant ideals, then $(A,G,\alpha)$ has the intersection property precisely when the reduced crossed product $A \times_r G$ is simple. In this case, the analysis required to apply Theorem \ref{thm:intro-main} is greatly simplified.

\begin{thm} \label{thm:intro-main-minimal}
Let $(A,G,\alpha)$ be a minimal C*-dynamical system. Suppose that for every pseudo-Glimm ideal $I \in \Glimm_p(A)$ and amenable subgroup $H \leq G_I$, there is a net $(s_i)$ in $G$ such that $\lim s_i H s_i^{-1} = \{e\}$ in the Chabauty topology. Then the reduced crossed product $A \times_r G$ is simple. If $(A,G,\alpha)$ has vanishing obstruction, then the converse is also true.
\end{thm}

Once again, if $A$ is prime, then the hypotheses of Theorem \ref{thm:intro-main-minimal} are easier to verify. 

\begin{cor}
Let $(A,G,\alpha)$ be a minimal C*-dynamical system such that $A$ is prime and let $G_A = \{s \in G : \alpha_s \text{ is quasi-inner} \}$. If for every amenable subgroup $H \leq G_A$, there is a net $(s_i)$ in $G$ such that $\lim s_i H s_i^{-1} = \{e\}$ in the Chabauty topology, then the reduced crossed product $A \times_r G$ is simple. If $G_A$ has vanishing Mackey obstruction $H^2(G_A,\mathbb{T})$, then the converse is also true.
\end{cor}

The theory of essential and injective extensions of C*-dynamical systems as introduced by Hamana \cite{H1985} plays an important role in the work of Kawabe and in the previous work of the first author. A key point is the duality between topological boundaries of $G$ in the sense of Furstenberg and essential extensions of the trivial C*-dynamical system $(\bC,G)$. In particular, the injective hull of $(\bC,G)$ is the C*-dynamical system $(C(\fb),G,\alpha)$, where $\fb$ denotes the universal topological boundary of $G$ and $\alpha$ denotes the action on $C(\fb)$ induced by the unique action of $G$ on $\fb$.

Similar ideas play an important role in our work. We work within the category of noncommutative dynamical systems, where a noncommutative dynamical system is a triple $(S,G,\alpha)$ consisting of an operator system $S$,  i.e. a unital self-adjoint subspace of a C*-algebra, a discrete group $G$ and a group homomorphism $\alpha : G \to \Aut(S)$, where $\Aut(S)$ denotes the group of order isomorphisms of $S$. Using ideas from the theory of matrix convexity, we associate a ``matrix affine'' dynamical system to every noncommutative dynamical system. We introduce a generalized notion of topological boundary for matrix affine dynamical systems and establish a duality with essential extensions of the corresponding noncommutative dynamical system.

In addition to this introduction, this paper has eight other sections. In Section 2, we introduce various definitions and results needed from the theory noncommutative dynamics, including the notion of the injective hull of a C*-dynamical system. In Section 3, we introduce the notion of a pseudo-Glimm ideal. In Section 4, we introduce the notion of a quasi-stabilizer subgroup and establish a new characterization of proper outerness for an automorphism of a C*-algebra. In Section 5, we briefly review the ideal separation property and the intersection property. In Section 6, we introduce a definition of pseudo-expectation for a C*-dynamical system and prove that pseudo-expectations can be used to detect the intersection property. In Section 7, we introduce the definition of a boundary for a matrix affine dynamical system associated to a noncommutative dynamical system. In Section 8, we associate a partial dynamical system to every C*-dynamical system and consider an associated cohomological obstruction. In Section 9 we establish our main results characterizing the the intersection property. 

\addtocontents{toc}{\setcounter{tocdepth}{-10}}
\section*{Acknowledgements}
\addtocontents{toc}{\setcounter{tocdepth}{2}}

The authors are grateful to Rob Archbold, Rasmus Bryder, Shirly Geffen, Douglas Somerset, Yuhei Suzuki and the anonymous referees for a number of helpful comments, corrections and suggestions.

\section{Noncommutative dynamics}

\subsection{Topological dynamical systems} \label{sec:topological-dynamical-systems}

A {\em topological dynamical system} is a triple $(X,G,h)$ consisting of a topological space $X$, a countable discrete group $G$ and an action of $G$ on $X$ in the form of a group homomorphism $h : G \to \operatorname{Homeo}(X)$. If $X$ is compact and Hausdorff, we will say that $(X,G,h)$ is a {\em compact dynamical system}. If we do not need to explicitly refer to the action $h$, then we will suppress $h$ and write $(X,G)$ for $(X,G,h)$.

A compact topological dynamical system $(X,G,h)$ is said to be {\em affine} if $X$ is a convex subset of a locally convex topological vector space and each homeomorphism $h_s$ is affine for $s \in G$.

Let $\P(X)$ denote the space of regular Borel probability measures on $X$ equipped with the weak* topology. A compact dynamical system $(X,G,h)$ gives rise to an affine dynamical system $(\P(X),G,k)$, where the action $k : G \to \operatorname{Homeo}(\P(X))$ is defined by $k_s(\mu)(E) = \mu(s^{-1}E)$ for $s \in G$ and $E \subseteq X$ Borel.

Consider the topological dynamical system $(X,G,h)$. If the orbit $Gx = \{sx : s \in G\}$ is dense in $X$ for every $x \in X$, then $(X,G,h)$ is said to be {\em minimal}.

For $x \in X$, the {\em neighborhood stabilizer subgroup} $G_x^\circ$ consists of the elements in $G$ that fix a neighborhood of $x$. If $G_x^\circ$ is trivial for all $x \in X$, then $(X,G,h)$ is said to be {\em topologically free}.

\subsection{C*-dynamical systems}

For a reference on C*-algebras and C*-dynamical systems, we refer the reader to the book of Pedersen \cite{P1979}. For a reference on crossed products, we refer the reader to the book of Brown and Ozawa \cite{BO2008} or the book of Williams \cite{W2007}.

A {\em C*-dynamical system} is a triple $(A,G,\alpha)$ consisting of a unital C*-algebra $A$, a countable discrete group $G$ an action of $G$ on $A$ in the form of a group homomorphism $\alpha : G \to \operatorname{Aut}(A)$.

We will say that $(A,G,\alpha)$ is {\em commutative} if $A$ is commutative. In this case, $A = C(X)$ for a compact Hausdorff space $X$, and by Gelfand duality we obtain a corresponding compact dynamical system $(X,G,h)$. Here the action $h : G \to \operatorname{Homeo}(X)$ is defined by $h_s = \alpha_s^*|_{X}$, for $s \in G$, where $\alpha_s^*$ denotes the adjoint of $\alpha_s$, and where we have identified points in $X$ with the corresponding pure states on $C(X)$. Note that by the Riesz-Markov-Kakutani representation theorem, the state space of $C(X)$ can be identified with the space $\P(X)$ of regular Borel probability measures on $X$.

We will say that a commutative C*-dynamical system $(C(X),G,\alpha)$ is {\em minimal} (resp. {\em topologically free}) if the corresponding compact dynamical system $(X,G,h)$ is minimal (resp. topologically free).

\subsection{Essentiality and injectivity}

Although we are primarily interested in C*-dynamical systems, we will need to work in the larger category of noncommutative dynamical systems. For a general reference on operator systems and operator spaces, we refer the reader to the book of Paulsen \cite{P2002} or the book of Pisier \cite{P2003}. The material in this section is due to Hamana \cite{H1985}, although the presentation given here is slightly different than Hamana's presentation.

A {\em noncommutative dynamical system} is a triple $(S,G,\alpha)$ consisting of an operator system $S$, i.e. a unital self-adjoint subspace of a C*-algebra, a countable discrete group $G$ and a group homomorphism $\alpha : G \to \Aut(S)$, where $\Aut(S)$ denotes the group of complete order isomorphisms (i.e. unital complete isometries) of $S$. If we do not need to explicitly refer to the action $\alpha$, then we will suppress $\alpha$ and write $(S,G)$ for $(S,G,\alpha)$.

We consider the {\em category of noncommutative dynamical systems}. If $(T,G,\beta)$ is a noncommutative dynamical system, then a morphism is an equivariant unital complete order preserving map (i.e. a unital completely positive map) $\phi : (S,G,\alpha) \to (T,G,\beta)$. It will be convenient to write this as $\phi : S \to T$ when the group and the actions are clear from the context. If, in addition, $\phi$ is a complete order injection, then we say that it is an {\em embedding} of $(S,G,\alpha)$ into $(T,G,\beta)$.

An {\em extension} of $(S,G,\alpha)$ is a pair $((T,G,\beta),\iota)$ consisting of a noncommutative C*-dynamical system $(T,G,\beta)$ and an embedding $\iota : S \to T$. We will write $(S,G,\alpha) \subseteq (T,G,\beta)$ when $S \subseteq T$, $\iota$ is the inclusion map and $\beta|_S = \alpha$. The extension is {\em essential} if whenever $(R,G,\gamma)$ is a noncommutative C*-dynamical system and $\phi : T \to R$ is a morphism such that $\phi|_S$ is an embedding, then $\phi$ is an embedding.

The noncommutative C*-dynamical system $(R,G,\gamma)$ is injective if whenever $((T,G,\beta),\iota)$ is an extension of $(S,G,\alpha)$ and $\phi : S \to R$ is a morphism, then there is a morphism $\psi : T \to R$ such that $\psi \circ \iota = \phi$.

Hamana \cite{H1985}*{Theorem 2.5} showed that every noncommutative dynamical system $(S,G,\alpha)$ has a minimal injective extension $((I_G(S),G,I_G(\alpha),\iota)$ that is unique up to isomorphism called the {\em injective hull} of $(S,G,\alpha)$. This extension is a C*-dynamical system, meaning in particular that $I_G(S)$ is a C*-algebra. If $S$ is an operator subsystem of a commutative C*-algebra, then $I_G(S)$ is commutative. We will identify $S$ with its canonical image in $I_G(S)$.

It follows from above that the action $I_G(\alpha)$ extends the action $\alpha$ and is uniquely determined by it. When there is no chance of confusion, it will be convenient to simply write $\alpha$ instead of $I_G(\alpha)$.

We may view an operator system $S$ as a noncommutative dynamical system with respect to the trivial group and the trivial action. In this case, the C*-algebra of the injective hull constructed as above coincides with Hamana's \cite{H1979} injective hull $I(S)$ of $S$ in the category of operator systems. In this category, the morphisms are unital completely positive maps. Hamana \cite{H1985}*{Remark 2.3} showed that $I_G(S)$ is an injective object in the category of operator systems for any discrete group $G$. This implies the inclusions $S \subseteq I(S) \subseteq I_G(S)$. However, typically $I(S) \ne I_G(S)$. 

Every automorphism of $S$ extends uniquely to an automorphism of $I(S)$, so we obtain a noncommutative dynamical system $(I(S),G,I(\alpha))$, along with the inclusion of noncommutative dynamical systems
\[
(S,G,\alpha) \subseteq (I(S),G,I(\alpha)) \subseteq (I_G(S),G,I_G(\alpha)).
\]
In particular, $(I(S),G,I(\alpha))$ is an essential extension of $(S,G,\alpha)$.

\begin{example}
Let $(M,G,\alpha)$ be a C*-dynamical system such that $M$ is an injective von Neumann algebra and $G$ is amenable. Then it follows from a result of Hamana \cite{H1985}*{Remark 3.8} that $(M,G,\alpha)$ is injective.
\end{example}

\subsection{Boundaries for affine dynamical systems}

A compact dynamical system $(X,G)$ is said to be a {\em boundary} for $G$ if for any probability measure $\mu \in \P(X)$, $X \subseteq \overline{G\mu}$. Here we have identified points in $X$ with the corresponding point masses in $\P(X)$. It is not difficult to see (cf. \cite{G1976}*{III.2}) that $(X,G)$ is a boundary for $G$ if and only if the corresponding affine dynamical system $(\P(X),G)$ is irreducible, meaning that there is no proper invariant compact convex subset of $\P(X)$.

A boundary $(X,G)$ for $G$ can be viewed as a boundary for the singleton dynamical system $(\{\epsilon\}, G)$, and from above, we can equivalently view the compact affine dynamical system $(\P(X),G)$ as a boundary for the singleton affine dynamical system $(\P(\{\epsilon\}),G)$. In this section we will define a notion of affine boundary for an arbitrary compact affine dynamical system. This will suggest an appropriate definition of a boundary for an arbitrary compact dynamical system.

\begin{defn} \label{defn:commutative-affine-boundary}
Let $(K,G)$ be a compact affine dynamical system. We say that a pair $((L,G),f)$ consisting of an affine dynamical system $(L,G)$ and a continuous surjective affine equivariant map $f : L \to K$ is a {\em boundary} for $(K,G)$ if whenever $L' \subseteq L$ is an invariant closed convex subset satisfying $f(L') = K$, then $L' = L$.
\end{defn}

\begin{rem}
For compact dynamical systems $(X,G)$ and $(Y,G)$ along with an equivariant continuous surjective map $f : Y \to X$, it would be natural to say that $((Y,G),f)$ is a boundary for $(X,G)$ if $(\P(Y),G),g)$ is a boundary for $(\P(X),G)$ in the sense of Definition \ref{defn:commutative-affine-boundary}, where $g : \P(Y) \to \P(X)$ denotes the map induced by $f$.
\end{rem}

The following C*-algebraic characterization of topological boundaries was a key result in \cite{KK2017}.

\begin{thm}
Let $(X,G)$ be a compact dynamical system. Then $(X,G)$ is a boundary for $G$ if and only if the inclusion of C*-dynamical systems $(\bC,G) \subseteq (C(X),G)$ is essential. In particular, the injective hull of $(\bC, G)$ is $(C(\fb), G)$, where $\fb$ denotes Furstenberg's universal boundary for $G$.
\end{thm}

We will see that an analogue of the above result holds for affine boundaries as defined in Definition \ref{defn:commutative-affine-boundary}.

For an affine dynamical system $(K,G,h)$, let $A(K)$ denote the space of continuous affine functions on $K$. Then $A(K)$ is an operator system and $(A(K),G,\alpha)$ is a noncommutative dynamical system, where $\alpha : G \to \Aut(A(K))$ is defined by $\alpha_s(a) = a \circ h_{s^{-1}}$ for $s \in G$ and $a \in A(K)$. By a result of Kadison \cite{K1951}, every operator subsystem of a commutative C*-algebra is isomorphic to the space of continuous affine functions on its state space. We require the following result  characterizing isomorphisms of operator systems of this form (see e.g. \cite{AS2001}*{Corollary 2.122}).

\begin{lem} \label{lem:adjoint-order-isomorphism-iff-homeomorphism}
Let $K$ and $L$ be compact convex sets, let $\phi : A(K) \to A(L)$ be a unital order preserving map and let $f : L \to K$ denote the restriction of the adjoint of $\phi$. Then $\phi$ is an order isomorphism if and only if $f$ is an affine homeomorphism. 
\end{lem}

\begin{thm} \label{thm:affine-boundary-iff-essential-extension}
Let $(K,G)$ and $(L,G)$ be affine dynamical systems and let $f : L \to K$ be an equivariant continuous surjective affine map. Then $((L,G),f)$ is a boundary for $(K,G)$ if and only if the corresponding extension $((A(L),G),\iota)$ of $(A(K),G)$ is essential, where $\iota : A(K) \to A(L)$ denotes the equivariant unital order preserving map defined by $\iota(a)(\mu) = a(f(\mu))$ for $a \in A(K)$ and $\mu \in L$.
\end{thm}

\begin{proof}
($\Rightarrow$)
Suppose $((A(L),G),\iota)$ is an essential extension of $(A(K),G)$ and let $L' \subseteq L$ be a closed convex invariant subset satisfying $f(L') = K$. Let $\rho : A(L) \to A(L')$ denote the restriction map. Then $\rho \circ \iota$ is an embedding, so by essentiality, $\rho$ is an order isomorphism. The restriction of the adjoint of $\rho$ to $L'$ is the inclusion map from $L'$ into $L$. Hence by Lemma \ref{lem:adjoint-order-isomorphism-iff-homeomorphism}, $L' = L$.

($\Leftarrow$)
Suppose $((L,G),f)$ is a boundary for $(K,G)$. Let $(N,G)$ be an affine dynamical system and let $\phi : A(L) \to A(N)$ be an equivariant unital order preserving map such that $\phi \circ \iota$ is an embedding. We can assume $\phi$ is surjective. Let $g : N \to L$ denote the restriction of the adjoint of $\phi$ and let $L' = g(N)$. Since $\phi \circ \iota$ is an embedding, it follows from Lemma \ref{lem:adjoint-order-isomorphism-iff-homeomorphism} that $f \circ g$ is an affine homeomorphism. In particular, $f(L') = K$. Since $((L,G),f)$ is a boundary for $(K,G)$, it follows that $L' = L$. Hence by Lemma \ref{lem:adjoint-order-isomorphism-iff-homeomorphism}, $\phi$ is an order isomorphism.
\end{proof}

The next result is an easy consequence of Theorem \ref{thm:affine-boundary-iff-essential-extension}.

\begin{cor} \label{cor:characterization-affine-boundaries}
Let $(C(X),G)$ and $(C(Y),G)$ be commutative C*-dynamical systems. Let $\iota : C(X) \to C(Y)$ be an equivariant embedding and let $f : \P(Y) \to \P(X)$ denote the restriction of the adjoint of $\iota$. Then $((C(Y),G),\iota)$ is an essential extension of $(C(X),G)$ if and only if whenever $K \subseteq \P(Y)$ is an invariant closed convex subset satisfying $f(K) = \P(X)$, then $K = \P(Y)$.
\end{cor}

\subsection{Properly outer and quasi-inner automorphisms} \label{sec:properly-outer-quasi-inner}

To study noncommutative C*-dynamical systems, it will be necessary to identify two special types of automorphisms. First, we need to identify the appropriate noncommutative analogue of a topologically free automorphism. Second, we need to identify automorphisms that are non-inner but ``close'' to being inner, for which there is no commutative analogue. Initially, we will characterize both types of automorphisms in terms of their Borchers spectrum. However, we will soon establish characterizations more useful for our purposes. For a definition of the Borchers spectrum, we direct the reader to Pedersen's book \cite{P1979}*{Section 8.8}.

\begin{defn} \label{defn:inner-outer}
Let $A$ be a unital C*-algebra and let $\alpha$ be an automorphism of $A$. We say that $\alpha$ is {\em properly outer} if for every $\alpha$-invariant ideal of $A$, the Borchers spectrum of the restriction $\alpha|_I$ satisfies $\Gamma_B(\alpha|_I) \ne \{1\}$.We say that $\alpha$ is {\em inner} if there is a unitary $u \in A$ such that $\alpha = \Ad(u)$, i.e. such that $\alpha(a) = uau^*$ for all $a \in A$. We say that $\alpha$ is {\em quasi-inner} if the Borchers spectrum of $\alpha$ satisfies $\Gamma_B(\alpha) = \{1\}$.
\end{defn}

\begin{rem}
The above definitions of proper outerness and quasi-innerness are due to Kishimoto \cite{K1982}. However, there are several different definitions of proper outerness in the literature.

The first definition of proper outerness was introduced by Kallmann \cite{K1969} under a different name. Kallmann says that $\alpha$ is properly outer if whenever $a_\alpha \in A$ satisfies $a a_\alpha = a_\alpha \alpha(a)$ for all $a \in A$, then $a_\alpha = 0$. If $A$ is a monotone complete C*-algebra, and in particular if $A$ is injective, then $\alpha$ is properly outer in the sense of Kishimoto if and only if it is properly outer in the sense of Kallmann \cite{H1982a}.

A definition of proper outerness appropriate for general C*-algebras was introduced by Elliott \cite{E1980}. Elliott says that $\alpha$ is properly outer if for every invariant ideal $I$ in $A$ and every unitary $u \in M(I)$, $\| \alpha|_I - \Ad(u)|_I \| = 2$. If $\alpha$ is properly outer in the sense of Kishimoto, then it is properly outer in the sense of Elliott, and if $A$ is separable, then the converse is also true \cite{OP1982}*{Theorem 6.6}.
\end{rem}

Hamana \cite{H1985}*{Theorem 7.4} characterized quasi-innerness in terms of the injective hull. Since the injective hull is a good approximation to $A$, in the sense that the inclusion $A \subseteq I(A)$ is essential, the next result says that quasi-inner automorphisms are close to being inner.

\begin{thm} \label{thm:inj-env-characterization-quasi-inner-properly-outer}
Let $A$ be a unital C*-algebra and let $\alpha : A \to A$ be an automorphism. Then $\alpha$ is quasi-inner (resp. properly outer) if and only if the unique extension of $\alpha$ to $I(A)$ is inner (resp. properly outer).
\end{thm}

We will make frequent use of the following two results (see \cite{H1982a}*{Proposition 5.1} and \cite{H1985}*{Remark 7.5}).

\begin{prop} \label{prop:largest-projection-inner-outer}
Let $A$ be a unital C*-algebra with injective hull $I(A)$. Let $\alpha$ be an automorphism of $A$ with unique extension $I(\alpha)$ to $I(A)$. There is a largest $I(\alpha)$-invariant projection $p_\alpha$ in $I(A)$ with the property that $I(\alpha)|_{p_\alpha I(A) p_\alpha}$ is inner. The projection $p_\alpha$ is central and $I(\alpha)|_{I(A)(1-p_\alpha)}$ is properly outer.
\end{prop}

\begin{prop} \label{prop:largest-ideal-inner-outer}
Let $A$ be a C*-algebra with injective hull $I(A)$. Let $\alpha$ be an automorphism of $A$ with unique extension $I(\alpha)$ to $I(A)$. There is a largest $\alpha$-invariant ideal $I_\alpha$ of $A$ such that $\alpha|_{I_\alpha}$ is quasi-inner and a largest $\alpha$-invariant ideal $J_\alpha$ of $A$ such that $\alpha|_{J_\alpha}$ is properly outer. These ideals satisfy $I_\alpha \cap J_\alpha = 0$, $I_\alpha^\perp = J_\alpha$ and $J_\alpha^\perp = I_\alpha$. Moreover, $I_\alpha + J_\alpha$ is essential. Let $p_\alpha \in I(A)$ denote the largest $\alpha$-invariant projection with the property that $I(\alpha)|_{I(A)p_\alpha}$ is inner as in Proposition \ref{prop:largest-projection-inner-outer}. Then $I_\alpha = A \cap I(A)p_\alpha$ and $J_\alpha = A \cap I(A)(1-p_\alpha)$.
\end{prop}

\begin{rem} \label{rem:quasi-inner-on-essential-ideal}
The above proposition implies that if $\alpha$ is inner, then it is quasi-inner. Furthermore, if the restriction $\alpha|_I$ is properly outer (resp. quasi-inner) for an invariant essential ideal $I$ in $A$, then  $\alpha$ is properly outer (resp. quasi-inner).
\end{rem}

For a commutative C*-dynamical system $(C(X),G,\alpha)$ with injective hull $(I_G(C(X)),G,I_G(\alpha))$, it is not hard to see that for $s \in G$, $I_G(\alpha)_s$ is topologically free if $\alpha_s$ is topologically free. The next result is a noncommutative generalization of this fact.

\begin{thm} \label{thm:prop-outer-implies-ess-prop-outer}
Let $(A,G,\alpha)$ be a C*-dynamical system with injective hull $(I_G(A),G,I_G(\alpha))$. For $s \in G$, let $p_s \in I(A)$ denote the largest $I(\alpha)_s$-invariant projection such that $I(\alpha)_s|_{I(A)p_s}$ is inner and let $u_s \in I(A)p_s$ be a unitary such that $I(\alpha)_s(bp_s) = u_s b u_s^*$ for $b \in I(A)$. Similarly, let $q_s \in I_G(A)$ denote the largest $I_G(\alpha)_s$-invariant projection such that $I_G(\alpha)_s|_{I_G(A)q_s}$ is inner and let $v_s \in I_G(A)q_s$ be a unitary such that $I_G(\alpha)_s(c) = v_s c v_s^*$ for $c \in I_G(A)$. Then for $t \in G$, $I(\alpha)_t(u_s)$ is a unitary in $I(A)p_{tst^{-1}}$ such that $I(\alpha)_{tst^{-1}}(b) = I(\alpha)_t(u_s) b  I(\alpha)_t(u_s)^*$ for $b \in I(A)$. Similarly, $I_G(\alpha)_t(v_s)$ is a unitary in $I_G(A)q_{tst^{-1}}$ such that $I_G(\alpha)_{tst^{-1}}(c) = I_G(\alpha)_t(v_s) c  I_G(\alpha)_t(u_s)^*$ for $c \in I_G(A)$. Hence $I(\alpha)_t(p_s) = p_{tst^{-1}}$ and $I_G(\alpha)_t(q_s) = q_{tst^{-1}}$. Furthermore, $p_s \geq q_s$, so if $I(\alpha)_s$ is properly outer, then so is $I_G(\alpha)_s$.
\end{thm}

\begin{proof}
Since $I(A)$ is injective, there is a conditional expectation $\phi : I_G(A) \to I(A)$. Define a $G$-equivariant unital completely positive map $\psi : I_G(A) \to \ell^\infty(G,I(A))$ by $\psi(b) = \sum_{t \in G}  \delta_t \otimes \phi(I_G(\alpha)_{t^{-1}}(c))$ for $c \in I_G(A)$. It is clear that the restriction $\psi|_{I(A)}$ is completely isometric. Hence by the $G$-injectivity of $I_G(A)$, $\psi$ is completely isometric.

Note that for $t \in G$, $I(\alpha)_t(p_s) = p_{tst^{-1}}$. Indeed, $b u_s = u_s I(\alpha)_s(b)$ for all $b \in I(A)$, so 
$I(\alpha)_{t^{-1}}(b) u_s = u_s I(\alpha)_{st^{-1}}(b)$ for all $b \in I(A)$, and applying $I(\alpha)_t$ to both sides gives $b I(\alpha)_t(u_s) = I(\alpha)_t(u_s) I(\alpha)_{tst^{-1}}(b)$. So $I(\alpha)_t(p_s) \leq p_{tst^{-1}}$. But then applying $I(\alpha)_{t^{-1}}$ to both sides and repeating the same argument implies $p_s \leq I(\alpha)_{t^{-1}}(p_{tst^{-1}}) \leq p_s$. Hence $I(\alpha)_t^{-1}(p_{tst^{-1}}) = p_s$, giving $p_{tst^{-1}} = I(\alpha)_t(p_s)$. Similarly, for $t \in G$, $q_{tst^{-1}} = I_G(\alpha)_t(q_s)$.

Since $I(A)$ is in the multiplicative domain of $\phi$,
\begin{align*}
b\phi(I_G(\alpha)_t(v_s)) &= \phi(b I_G(\alpha)_t(v_s)) \\
&= \phi(I_G(\alpha)_t(v_s) I(\alpha)_{tst^{-1}}(b)) \\
&= \phi(I_G(\alpha)_t(v_s)) I(\alpha)_{tst^{-1}}(b)
\end{align*}
for all $t \in G$ and $b \in I(A)$.

Therefore, $\phi(I_G(\alpha)_t(v_s)) \in I(A)(\alpha)_t(p_s)$. It follows that
\[
0 = \phi(I_G(\alpha)_t(v_s)) (1-I(\alpha)_t(p_s))  = \phi(I_G(\alpha)_t(v_s (1-p_s)))
\]
for all $t \in G$. Hence $\psi(v_s (1-p_s)) = 0$, and since $\psi$ is completely isometric, it follows that $v_s (1-p_s) = 0$.
\end{proof}

\begin{defn} \label{defn:properly-outer-quasi-inner-action}
Let $(A,G,\alpha)$ be a C*-dynamical system. We say that $(A,G,\alpha)$ is {\em properly outer} if for every $s \in G \setminus \{e\}$, $\alpha_s$ is properly outer. We say that $(A,G,\alpha)$ is {\em inner} if there is a group homomorphism $u : G \to U(A)$ such that for every $s \in G$, $\alpha_s = \Ad(u_s)$. Here $U(A)$ denotes the unitary group of $A$.
\end{defn}

The next result follows immediately from Theorem \ref{thm:prop-outer-implies-ess-prop-outer} and Definition \ref{defn:properly-outer-quasi-inner-action}.

\begin{cor}
Let $(A,G)$ be a C*-dynamical system with injective hull $(I_G(A),G)$. If $(A,G)$ is properly outer, then $(I_G(A),G)$ is properly outer.
\end{cor}

\section{Glimm ideals and pseudo-Glimm ideals} \label{sec:injectivity-and-essentiality}

\subsection{Glimm ideals} \label{sec:glimm-ideals}

\begin{defn}
Let $A$ be a unital C*-algebra with center $Z(A)$. The {\em Glimm space} $\Glimm(A)$ is the spectrum of $Z(A)$.
\end{defn}

For a unital C*-algebra $A$, there are several equivalent characterizations of $\Glimm(A)$. For example, it can be identified with the complete regularization (equivalently, the Hausdorffization) of the primitive ideal space $\Prim(A)$ (see e.g. \cite{RW1998}*{A.3}). We require a characterization of $\Glimm(A)$ in terms of Glimm ideals.

\begin{defn}
Let $A$ be a unital C*-algebra with center $Z(A)$. An ideal of $A$ is a {\em Glimm ideal} if it be written as $AK_x$ for some $x \in \Glimm(A)$, where $K_x = \{ f \in Z(A) : f(x) = 0 \}$ is the maximal ideal in $Z(A)$ corresponding to $x$.
\end{defn}

\begin{rem}
It follows from Cohen's factorization theorem that if $K$ is an ideal of $Z(A)$, then $AK$ is a (closed) ideal of $A$.
\end{rem}

For $x \in \Glimm(A)$ with corresponding Glimm ideal $AK_x$ as above, $Z(A) \cap AK_x = K_x$. Thus we obtain a bijective correspondence between points in $\Glimm(A)$ and Glimm ideals in $A$. It will be convenient to identify these spaces and write $x$ for both the point $x \in \Glimm(A)$ and the ideal $AK_x$ in $A$.

For a C*-dynamical system $(A,G,\alpha)$, we obtain a compact topological dynamical system $(\Glimm(A),G,h)$, where $h$ is the action induced by $\alpha$.

Hamana \cite{H1985}*{Proposition 6.4} proved the inclusions
\[
(Z(A),G) \subseteq (Z(I(A)),G) \subseteq (Z(I_G(A)),G).
\]
We identify $\Glimm(A)$, $\Glimm(I(A))$ and $\Glimm(I_G(A))$ with the spectrum of the center of $A$, $I(A)$ and $I_G(A)$ respectively. Applying Gelfand duality, we obtain continuous equivariant surjective maps:
\begin{equation*}
\begin{tikzcd}
\Glimm(A) & \arrow{l}{f} \Glimm(I(A)) & \arrow{l}{g} \arrow[bend right]{ll}{f \circ g} \Glimm(I_G(A))
\end{tikzcd}
\end{equation*}
The map $f$ can be characterized as the restriction to $\Glimm(I(A))$ of the intersection map between the set of ideals of $Z(I(A))$ and the set of ideals of $Z(A)$. Specifically, for $y \in \Glimm(I(A))$, let $x = f(y)$. Then letting $K_x$ and $K_y$ denote the corresponding maximal ideals in $Z(A)$ and $Z(I(A))$ respectively, $K_x = Z(A) \cap K_y$. The maps $g$ and $f \circ g$ have similar characterizations.

We suspect that the following result is well known.

\begin{prop} \label{prop:glimm-injective-hull-extremally-disconnected}
Let $B$ be an injective C*-algebra. Then $\Glimm(B)$ is extremally disconnected, i.e. the closure of every open subset is clopen.
\end{prop}

\begin{proof}
By \cite{S1949} (see also \cite{T2002}*{Proposition 1.7}) it suffices to show that $Z(B)$ is monotone complete, i.e. that for every bounded increasing net $(b_i)$ in $Z(B)$, $\sup_{Z(B)} b_i$ exists and belongs to $Z(B)$. Since $B$ is injective it is monotone complete. Let $b = \sup_B b_i$. Then for every unitary $u \in B$, $ubu^* = \sup_B u b_i u^* = \sup_B b_i = b$. Hence $b \in Z(B)$ and clearly $b = \sup_{Z(B)} b_i$.
\end{proof}

It follows from a result of Gleason \cite{G1958} that if $X$ is a compact Hausdorff space, then the commutative C*-algebra $C(X)$ is injective if and only if $X$ is extremally disconnected. Thus Proposition \ref{prop:glimm-injective-hull-extremally-disconnected} implies that the center of every injective C*-algebra is injective. However, Hamana \cite{H1982b}*{Corollary 1.6} observes that in some cases $I(Z(A)) \ne Z(I(A))$ (see Example \ref{ex:pseudo-glimm}). See the recent paper of Bryder \cite{B2017}*{Remark 4.13} for an equivariant example of this phenomenon.

\subsection{Pseudo-Glimm ideals}

Let $A$ be a unital C*-algebra and let $\Id(A)$ denote the space of ideals of $A$ equipped with the compact Hausdorff {\em strong topology} (or Fell topology) \cite{F1962} (see also \cite{A1987}). A net $(I_i)$ in $\Id(A)$ converges to $I \in \Id(A)$ in this topology if and only if $\lim \| a + I_i \| = \|a + I\|$. We will frequently use the fact that for an ideal $I \in \Id(A)$, the set $\{ I' \in \Id(A) : I' \not \supseteq I \}$ is open.

\begin{lem} \label{lem:intersection-map-continuous}
Let $A$ and $B$ be unital C*-algebras with $A \subseteq B$. Then the intersection map from $\Id(B)$ to $\Id(A)$ is continuous in the strong topology.
\end{lem}

\begin{proof}
Let $(J_i)$ be a net of ideals in $B$ converging to an ideal $J$ in $B$ in the strong topology. Then for $b \in B$,
\[
\lim \|b + J_i\| = \|b + J\|.
\]
Since $A/(A \cap J) \simeq A + J / J \subseteq B/J$, it follows that for $a \in A$,
\[
\lim \|a + A \cap J_i\| = \|a + A \cap J\|.
\]
Hence $A \cap J_i$ converges to $A \cap J$.
\end{proof}

An ideal $I$ in $A$ is {\em primal} if for all $n \geq 2$ and ideals $I_1,\ldots,I_n$ in $A$ with $I_1 \cdots I_n = 0$ there is $1 \leq i \leq n$ such that $I_i \subseteq I$. Note that the ideal $A$ is always primal. A primal ideal $I$ is {\em proper} if $I \ne A$.

Since $A$ is unital, it follows from \cite{A1987}*{Proposition 4.1} that the space of proper primal ideals is compact in the strong topology. By \cite{AS1990}*{Theorem 2.2}, if $I$ is a proper primal ideal in $A$, then there is a unique Glimm ideal $x \in \Glimm(A)$ such that $I \supseteq x$. 

\begin{defn}
Let $A$ be a unital C*-algebra with injective hull $I(A)$. An ideal $I$ of $A$ is a {\em pseudo-Glimm ideal} if there is a proper primal ideal $J$ in $I(A)$ such that $I = A \cap J$. We write $\Glimm_p(A)$ for the space of all pseudo-Glimm ideals of $A$ equipped with the relative strong topology.
\end{defn}

\begin{lem}
Let $A$ be a unital C*-algebra. The space $\Glimm_p(A)$ of pseudo-Glimm ideals of $A$ is compact in the strong topology.
\end{lem}

\begin{proof}
By \cite{A1987}*{Proposition 4.1}, the set of proper primal ideals of $I(A)$ is compact in the strong topology. The space of pseudo-Glimm ideals $\Glimm_p(A)$ of $A$ is the image of the set of proper primal ideals in $I(A)$ under the intersection map from $\Id(I(A))$ to $\Id(A)$. The result now follows from Lemma \ref{lem:intersection-map-continuous}, which implies this map is continuous in the strong topology.
\end{proof}

We will frequently use the following result.

\begin{prop} \label{prop:injective-is-quasi-standard}
Let $(A,G)$ be a C*-dynamical system. Then every Glimm ideal in $I(A)$ and $I_G(A)$ is proper and primal. Moreover, the strong topology and the quotient topology agree on $\Glimm(I(A))$ and $\Glimm(I_G(A))$. 
\end{prop}

\begin{proof}
Since $I(A)$ and $I_G(A)$ are injective, the results in \cite{AM1994} imply that they are quasi-standard. The result now follows from  \cite{AS1990}*{Theorem 3.3}.
\end{proof}

\begin{prop} \label{prop:inclusion-pseudo-glimm-ideals}
Let $(A,G)$ be a C*-dynamical system. Let $f : \Glimm(I(A)) \to \Glimm(A)$ and $g : \Glimm(I_G(A)) \to \Glimm(I(A))$ denote the maps induced by the inclusions $Z(A) \subseteq Z(I(A)) \subseteq Z(I_G(A))$. Then for $x \in \Glimm(A)$,
\[
x = \bigcap_{y \in f^{-1}(x)} A \cap y = \bigcap_{z \in (f \circ g)^{-1}(x)} A \cap z.
\]
\end{prop}

\begin{proof}
We will require the fact from Proposition \ref{prop:injective-is-quasi-standard} that the strong topology and the quotient topology agree on $\Glimm(I(A))$ and $\Glimm(I_G(A))$.

Fix $x \in \Glimm(A)$ and let $K_x$ denote the ideal in $Z(A)$ corresponding to $x$. Similarly, for $y \in f^{-1}(x)$, let $K_y$ denote the ideal in $Z(I(A))$ corresponding to $y$. Then $K_x = Z(A) \cap K_y$. Since $x = AK_x$ and $y = I(A)K_y$, it follows that $x \subseteq A \cap y$. Hence $x \subseteq \bigcap_{y \in f^{-1}(x)} A \cap y$.

For the reverse inclusion, suppose $a \in \bigcap_{y \in f^{-1}(x)} A \cap y$ with $\|a\| = 1$. For $\epsilon > 0$, it follows from \cite{P1979}*{Proposition 4.4.4} that the set
\[
F_\epsilon = \{y \in \Glimm(I(A)) : \|a + y\| \geq \epsilon \}
\]
is compact, and by construction it is disjoint from $f^{-1}(x)$. By the continuity of $f$, $f(F_\epsilon)$ is a closed set disjoint from $x$. Hence by Urysohn's Lemma there is $a' \in Z(A)$ with $0 \leq a' \leq 1$ such that $a'(x') = 1$ for $x' \in f(F_\epsilon)$ and $a'(x) = 0$. Hence $\|aa' - a\| < \epsilon$. It follows that $a \in x$ and we conclude that $x = \bigcap_{y \in f^{-1}(x)} A \cap y$.

The argument for $\Glimm(I_G(A))$ is similar.
\end{proof}

\begin{prop} \label{prop:intersections-glimm-ideals-are-pseudo-glimm}
Let $(A,G)$ be a C*-dynamical system. Every Glimm ideal $y \in \Glimm(I(A))$ is proper and primal. Hence the intersection $A \cap y$ is a pseudo-Glimm ideal in $A$. Similarly, for every Glimm ideal $z \in \Glimm(I_G(A))$, the intersection $I(A) \cap z$ is a proper primal ideal in $I(A)$. Hence the intersection $A \cap z$ is a pseudo-Glimm ideal in $A$.
\end{prop}

\begin{proof}
By Proposition \ref{prop:injective-is-quasi-standard}, every Glimm ideal in $I(A)$ is proper and primal. Fix a Glimm ideal $z \in \Glimm(I_G(A))$. Since $z$ is a proper ideal, $I(A) \cap z$ is a proper ideal in $I(A)$. By Proposition \ref{prop:inclusion-pseudo-glimm-ideals}, there is a Glimm ideal $y \in \Glimm(I(A))$ such that $I(A) \cap z \supseteq y$. From above, $y$ is primal. Hence letting $J_1,\ldots,J_n$ be ideals in $I(A)$ such that $J_1 \cdots J_n = 0$, there is $1 \leq i \leq n$ such that $J_i \subseteq y \subseteq I(A) \cap z$. It follows that $I(A) \cap z$ is primal. Hence $A \cap z = A \cap (I(A) \cap z)$ is a pseudo-Glimm ideal in $A$.
\end{proof}

The proof of Proposition \ref{prop:inclusion-pseudo-glimm-ideals} provides more information about pseudo-Glimm ideals.

\begin{lem} \label{lem:inclusion-pseudo-glimm-ideals}
Let $A$ be a unital C*-algebra. Then every pseudo-glimm ideal in $A$ contains a Glimm ideal in $A$, and every Glimm ideal in $A$ is contained in a pseudo-Glimm ideal in $A$.
\end{lem}

\begin{proof}
 Let $f : \Glimm(I(A)) \to \Glimm(A)$ denote the map induced by the inclusions $Z(A) \subseteq Z(I(A))$.
 
 Let $I \in \Glimm_p(A)$ be a pseudo-Glimm ideal. Then there is a proper primal ideal $J$ in $I(A)$ such that $I = A \cap J$. Let $y \in \Glimm(I(A))$ be the unique Glimm ideal contained in $J$. Then $x = f(y) \in \Glimm(A)$. Arguing as in the proof of Proposition \ref{prop:inclusion-pseudo-glimm-ideals}, we obtain $x \subseteq A \cap y \subseteq A \cap J = I$.
 
On the other hand, let $x \in \Glimm(A)$ be a Glimm ideal. Since the map $f$ is surjective, there is $y \in \Glimm(I(A))$ such that $x = f(y)$. Arguing as in the proof of Proposition \ref{prop:inclusion-pseudo-glimm-ideals}, $x \supseteq A \cap y$. 
\end{proof}

The motivation for the name ``pseudo-Glimm ideal'' is based on the fact that in many cases, pseudo-Glimm ideals are actually Glimm ideals. For example, the next result shows that this is the case if $A$ has Hausdorff primitive ideal space.

\begin{prop}
Let $A$ be a unital C*-algebra with Hausdorff primitive ideal space. Then the pseudo-Glimm ideals are precisely the Glimm ideals, i.e. $\Glimm_p(A) = \Glimm(A)$.
\end{prop}

\begin{proof}
If the primitive ideal space of $A$ is Hausdorff, then every Glimm ideal of $A$ is maximal (see e.g. \cite{MR2007}). For a proper primal ideal $J$ in $I(A)$, Lemma \ref{lem:inclusion-pseudo-glimm-ideals} implies that $A \cap J$ is a proper ideal in $A$ that contains a Glimm ideal, say $x$. By the maximality of $x$ it follows that $A \cap J = x$. Hence every pseudo-Glimm ideal is a Glimm ideal. The fact that every Glimm ideal is a pseudo-Glimm follows from the maximality of Glimm ideals and Lemma \ref{lem:inclusion-pseudo-glimm-ideals}.
\end{proof}

As the next example demonstrates, there are C*-algebras with pseudo-Glimm ideals that are not Glimm ideals. This phenomenon is closely related to the fact that the center of the injective envelope of $A$ is not always equal to the injective envelope of the center of $A$ (see e.g. \cite{H1982b} and \cite{S1982}). We are grateful to R.S. Bryder for pointing this out.

\begin{example} \label{ex:pseudo-glimm}
Let $H$ be an infinite dimensional Hilbert space and let $\mathcal{K}(H)$ denote the C*-algebra of compact operators on $H$. Let $A = \bC1_{H \oplus H} + \mathcal{K}(H) \oplus \mathcal{K}(H)$. Then it is easy to see that $I(A) = \B(H) \oplus \B(H)$. Hence
\[
Z(I(A)) = \bC1_H \oplus \bC1_H,
\]
while
\[
I(Z(A)) = Z(A) = \bC1_{H \oplus H}.
\]
In particular, $I(Z(A)) \subsetneq Z(I(A))$. 

The Glimm space of $I(A)$ is $\Glimm(I(A)) = \{0\}$. However, the proper primal ideals in $I(A)$ are $\B(H) \oplus \K(H)$, $\K(H) \oplus \B(H)$, $\B(H) \oplus 0$ and $0 \oplus \B(H)$, so the pseudo-Glimm space of $A$ is
\[
\Glimm_p(A) = \{\K(H) \oplus \K(H), \K(H) \oplus 0, 0 \oplus \K(H) \}.
\]
\end{example}

\begin{prop} \label{prop:quasi-glimm-ideals-covering}
Let $(A,G)$ be a C*-dynamical system with injective hull $(I_G(A),G)$. If $Z \subseteq \Glimm(I_G(A))$ is an invariant closed subset with the property that
\[
\bigcap_{z \in Z} A \cap z = 0,
\]
then $Z = \Glimm(I_G(A))$.
\end{prop}

\begin{proof}
Let $Z \subseteq \Glimm(I_G(A))$ be an invariant closed subset with the property that $\cap_{z \in Z} A \cap z = 0$. Let $K$ denote the invariant ideal in $I_G(A)$ defined by $K = \cap_{z \in Z} z$. Then
\[
A \cap K = \bigcap_{z \in Z} A \cap z = 0.
\]
By essentiality, it follows that $K = 0$, and we conclude that $Z = \Glimm(I_G(A))$.
\end{proof}

\section{Quasi-stabilizer subgroups}

\begin{defn}
Let $(A,G)$ be a C*-dynamical system. For $s \in G$, let $p_s \in I(A)$ denote the largest $I(\alpha)_s$-invariant projection in $I(A)$ such that $I(\alpha)_s|_{I(A)p_s}$ is inner. For a pseudo-Glimm ideal $I \in \Glimm_p(A)$, the {\em quasi-stabilizer subgroup} $G_I \leq G$ is defined by
\begin{align*}
G_I = \{s \in G : 1-p_s \in J &\text{ for every proper primal ideal } J \text{ in } I(A)  \\
&\text{ with } I = A \cap J \}.
\end{align*}
\end{defn}

\begin{rem}
Let $(A,G)$ be a C*-dynamical system and let $I$ be a pseudo-Glimm ideal in $A$. To see that the corresponding quasi-stabilizer $G_I$ is actually a group, suppose $s,t \in G_I$. If $J$ is a proper primal ideal in $I(A)$ with $I = A \cap J$, then $1-p_s,1-p_t \in J$. Since $p_s p_t \leq p_{st}$,
\[
1-p_{st} \leq 1-(p_s p_t) = (1-p_s) + (1-p_t) - (1-p_s) (1-p_t) \in J.
\]
In particular, $1-p_{st} \in J$ and hence $st \in G_I$.
\end{rem}

\begin{lem} \label{lem:only-one-inclusion-must-hold}
Let $(B,G)$ be a C*-dynamical system with $B$ injective. For $s \in G$, let $p_s \in B$ denote the largest $B_s$-invariant projection in $B$ such that $B_s|_{B p_s}$ is inner. Then for a proper primal ideal $J$ in $B$, exactly one of $p_s$ or $1-p_s$ belongs to $J$.
\end{lem}

\begin{proof}
The fact that $J$ is proper implies that at most one of $p_s$ or $1-p_s$ belongs to $J$. Let $y \in \Glimm(B)$ be the unique Glimm ideal contained in $J$. Then exactly one of $y \in \supp(p_s)$ or $y \in \supp(1-p_s)$ holds. In the first case $1-p_s \in y \subseteq J$, and in the second case $p_s \in y \subseteq J$.
\end{proof}

\begin{rem} \label{ref:quasi-stabilizer-injective}
Let $(B,G,\beta)$ be a C*-dynamical system such that $B$ is injective. It follows from Proposition \ref{prop:injective-is-quasi-standard} that every Glimm ideal in $B$ is pseudo-Glimm. For $s \in G$, let $p_s \in B$ denote the largest $\beta_s$-invariant projection such that $\beta_s|_{Bp_s}$ is inner. For $y \in \Glimm(B)$, $1-p_s \in Bp_s$ if and only if $y \in \supp(p_s)$. Thus the corresponding quasi-stabilizer subgroup is
\[
G_y = \{s \in G : y \in \supp(p_s) \}.
\]
\end{rem}

\begin{prop}
Let $(A,G)$ be a C*-dynamical system. For $s \in G$, $\alpha_s$ is properly outer if and only if $s \notin G_I$ for every pseudo-Glimm ideal $I \in \Glimm_p(A)$.
\end{prop}

\begin{proof}
Fix $s \in G$, let $p_s \in I(A)$ denote the largest $I(\alpha)_s$-invariant projection in $I(A)$ such that $I(\alpha)_s|_{I(A)p_s}$ is inner. Suppose $\alpha_s$ is properly outer, so that $p_s = 0$. Then for every proper primal ideal $J$ in $I(A)$, $1 = 1-p_s \notin J$. Hence for any pseudo-Glimm ideal $I$ in $\Glimm_p(A)$, $s \notin G_I$.

Conversely, suppose $s \notin G_I$ for every pseudo-Glimm ideal $I \in \Glimm_p(A)$. Then for every Glimm ideal $y \in \Glimm(I(A))$, there is a proper primal ideal $J$ in $I(A)$ such that $A \cap y = A \cap J$ and $1-p_s \notin J$. Then by Lemma \ref{lem:only-one-inclusion-must-hold}, $p_s \in J$, so $A \cap y = A \cap J \supseteq A \cap  I(A)p_s$. Since this holds for all $y \in \Glimm(I(A))$, it follows that
\[
0 = \bigcap_{y \in \Glimm(I(A))} A \cap y \supseteq A \cap I(A)p_s.
\]
Hence $A \cap I(A)p_s = 0$ and $\alpha_s$ is properly outer.
\end{proof}

\begin{prop} \label{prop:quasi-stabilizer-for-injective-is-amenable}
Let $(A,G)$ be a C*-dynamical system with injective hull $(I_G(A),G)$. Then for $z \in \Glimm(I_G(A))$, the quasi-stabilizer subgroup $G_z$ is amenable.
\end{prop}

\begin{proof}
For $s \in G$, let $q_s \in I_G(A)$ denote the largest $I_G(\alpha)_s$-invariant projection such that $I_G(\alpha)_s|_{I_G(A)q_s}$ is inner. Let $v_s \in I_G(A)q_s$ be a unitary such that $I_G(\alpha)_s|_{I_G(A)q_s} = \Ad(v_s)$.

For $s \in G_z$, $1 - q_s \in z$. Hence for $b \in I_G(A)$,
\[
I_G(\alpha)_s(b) + z = I_G(\alpha)_s(bq_s) + z = v_s b v_s^* + z.
\]

Let $\sigma_z : I_G(A) \to \B(H_z)$ be a representation of $I_G(A)$ with kernel $z$. Fix a unit vector $\xi \in H_z$ and define a state $\theta$ on $I_G(A)$ by
\[
\theta(b) = \langle \sigma_z(b)\xi, \xi \rangle
\]
for $b \in I_G(A)$.

Consider the C*-dynamical system $(\ell^\infty(G,I_G(A)),G)$. By the injectivity of $(I_G(A),G)$ there is an equivariant unital completely positive map $\phi : \ell^\infty(G,I_G(A)) \to I_G(A)$ such that $\phi|_{I_G(A)} = \id|_{I_G(A)}$. Then $I_G(A)$ belongs to the multiplicative domain of $\phi$, so for $f \in \ell^\infty(G)$ and $b \in I_G(A)$,
\[
b\phi(f) = \phi(bf) = \phi(fb) = \phi(f)b.
\]
Hence $\phi$ maps $\ell^\infty(G)$ into the center $Z(I_G(A))$ of $I_G(A)$.

The composition $\theta \circ \phi$ is a $G_z$-invariant state on $\ell^\infty(G)$, implying that $G_z$ is amenable.
\end{proof}

\begin{lem} \label{lem:monotone-projections}
Let $(A,G)$ be a C*-dynamical system. Let $p \in Z(I(A))$ and $q \in Z(I_G(A))$ be projections. Suppose
\[
A \cap I_G(A)q \supseteq A \cap I(A)p.
\]
Then $q \geq p$.
\end{lem}

\begin{proof}
Let $\bar{A}$ denote Hamana's regular monotone completion of $A$ \cite{H1981}. Then $A \subseteq \bar{A} \subseteq I(A)$ and by \cite{H1982b}*{Remark (b)}, $Z(\bar{A}) = Z(I(A))$. In particular, $p \in Z(\bar{A})$, so $p$ is open in $\bar{A}$, meaning there is an increasing net $(a_i)$ of self-adjoint elements in $A$ such that $p = \sup_{\bar{A}} a_i$. By \cite{H1981}*{Theorem 3.1} and \cite{H1985}*{Lemma 3.1}, the inclusions $A \subseteq \bar{A} \subseteq I(A) \subseteq I_G(A)$ are normal, so $p = \sup_{I_G(A)} a_i$.

For each $i$, $a_i \leq p$, so $a_i \in A \cap I(A)p \subseteq A \cap I_G(A)q$. Hence $a_i (1-q) = 0$, implying $a_i \leq p q$. Since $p = \sup_{I_G(A)} a_i$, it follows that $p \leq pq$, which implies $p \leq q$.
\end{proof}

\begin{lem} \label{lem:glimm-map-empty-interior}
Let $(A,G)$ be a C*-dynamical system. For $s \in G$, let $p_s \in I(A)$ denote the largest $I(\alpha)_s$-invariant projection such that $I(\alpha)_s|_{I(A)p_s}$ is inner. Then the subset
\[
\{z \in \Glimm(I_G(A)) : A \cap z \supseteq A \cap I(A)p_s \vee A \cap I(A)(1-p_s) \}
\]
is closed and has empty interior in $\Glimm(I_G(A))$.
\end{lem}

\begin{proof}
Let
\[
F = \{z \in \Glimm(I_G(A)) : A \cap z \supseteq A \cap I(A)p_s \vee A \cap I(A)(1-p_s) \}.
\]
Since the intersection map from $\Glimm(I_G(A))$ to $\Glimm_p(A)$ is strongly continuous, and since the set
\[
\{I \in \Id(A) : I \supseteq A \cap I(A)p_s \vee A \cap I(A)(1-p_s) \}
\]
is strongly closed, it follows that $F$ is closed.

Let $V$ denote the interior of $F$. By Proposition \ref{prop:glimm-injective-hull-extremally-disconnected}, $\Glimm(I_G(A))$ is extremally disconnected, so the closure $\overline{V}$ is clopen. Let $q \in Z(I_G(A)))$ denote the projection with support $\overline{V}$ and let $J$ denote the ideal in $I_G(A)$ defined by $J = \cap_{z \in \overline{V}} z$. Then $J = I_G(A)(1-q)$, so
\[
A \cap I_G(A)(1-q) = A \cap J = \bigcap_{z \in \overline{V}} A \cap z \supseteq A \cap I(A)p_s \vee A \cap I(A)(1-p_s).
\]
By Lemma \ref{lem:monotone-projections}, it follows that $p_s q = 0$ and $(1-p_s) q = 0$. Hence $q = 0$, so $V = \emptyset$ and we conclude that $F$ has empty interior.
\end{proof}

For a group $G$, we let $\Sub(G)$ denote the space of subgroups of $G$ equipped with the Chabauty topology (see Section \ref{sec:intrinsic-general}).

\begin{prop} \label{prop:continuous-dense-g-delta}
Let $(A,G)$ be a C*-dynamical system with injective hull $(I_G(A),G)$. There is a dense $G_\delta$ subset $L \subseteq \Glimm(I_G(A))$ with the property that $G_z \leq G_{A \cap z}$ for $z \in L$ and the stabilizer map
\[
\Glimm(I_G(A)) \to \Sub(G) : z \to G_z
\]
is continuous on $L$. If $(I_G(A),G)$ is not properly outer, then the set $\{z \in L : s \in G_z\}$ is non-empty. 
\end{prop}

\begin{proof}
The Chabauty topology on $\Sub(G)$ agrees with the relative product topology on $\Sub(G)$ when $\Sub(G)$ is identified with a subspace of $\{0,1\}^G$.

Define $r : \Glimm(I_G(A)) \to \Sub(G)$ by $r(z) = G_z$. Then $r$ is continuous at $z \in \Glimm(I_G(A))$ if and only if for every $s \in G$ the map $r_s : \Glimm(I_G(A)) \to \{0,1\}$ defined by
\[
r_s(z) =
\begin{cases}
1 & s \in G_z \\
0 & \text{otherwise}
\end{cases}
\]
is continuous at $z$.

For $s \in G$, let
\[
U_s = \{z \in \Glimm(I_G(A)) : A \cap z \not \supseteq A \cap I(A)p_s \vee A \cap I(A)(1-p_s) \}.
\]
By Lemma \ref{lem:glimm-map-empty-interior}, $U_s$ is open and dense. Fix $z \in U_s$ and let $q_s \in I_G(A)$ denote the largest $I_G(\alpha)_s$-invariant projection such that $I_G(\alpha)_s|_{I_G(A)q_s}$ is inner. Then $s \in G_z$ if and only if $z \in \supp(q_s)$, which is equivalent to $1-q_s \in z$.

By Lemma \ref{lem:only-one-inclusion-must-hold}, exactly one of $p_s \in z$ or $1-p_s \in z$ holds, so by the definition of $U_s$, exactly one of $A \cap z \supseteq A \cap I(A)p_s$ or $A \cap z \supseteq A \cap I(A)(1-p_s)$ holds.

If $s \in G_{A \cap z}$, then $1-p_s \in I(A) \cap z$. In this case, $A \cap z \supseteq A \cap I(A)(1-p_s)$, so $A \cap z \not \supseteq A \cap I(A) p_s$. Conversely, if $A \cap z \not \supseteq A \cap I(A) p_s$, then for any proper primal ideal $J$ in $I(A)$ with $A \cap J = A \cap z$, $p_s \not \in J$. By Lemma \ref{lem:only-one-inclusion-must-hold}, $1-p_s \in J$, and hence $s \in G_{A \cap z}$. Therefore, $s \in G_{A \cap z}$ if and only if $A \cap z \not \supseteq A \cap I(A)p_s$, or equivalently if and only if $A \cap z \supseteq A \cap I(A)(1-p_s)$. 

If $s \in G_z$, then from above $1-q_s \in z$, so by Theorem \ref{thm:prop-outer-implies-ess-prop-outer}, $1-p_s \in z$, implying $A \cap z \supseteq A \cap I(A)(1-p_s)$. It follows that in this case, $s \in G_{A \cap z}$, so $G_z \subseteq G_{A \cap z}$.

Now $r_s(z) = 1$ for all $z$ in the set $U_s \cap \supp(q_s))$, and $r_s(z) = 0$ for all $z$ in the set $U_s \cap \supp(1-q_s)$. These sets are open, since $\supp(q_s)$ and $\supp(1-q_s)$ are clopen. Moreover, they partition $U_s$. Therefore, $r_s$ is continuous on the dense open subset $U_s$. Since $G$ is countable, it follows from the Baire category theorem that the function $r$ is continuous on a dense $G_\delta$ subset $L \subseteq \Glimm(I_G(A))$.

For $s \in G$ and $z \in \Glimm(I_G(A))$, $s \in G_z$ if and only if $z \in \supp(q_s)$. If $I_G(\alpha)_s$ is not properly outer, then $\supp(q_s) \ne \emptyset$. Since $\supp(q_s)$ is open and $L$ is dense, $\supp(q_s) \cap L \ne \emptyset$.
\end{proof}

\section{Ideal separation and the intersection property}
For a C*-dynamical system $(A,G)$, recall that $\Id(A)$ denotes the set of (closed two-sided) ideals in $A$ and $\Id(A)^G$ denotes the set of invariant ideals in $A$. The following definition was introduced by Sierakowski \cite{S2010}.

\begin{defn}
A C*-dynamical system $(A,G)$ is said to have the {\em ideal separation property} if the intersection map
\[
\Id(A \times_r G) \to \Id^G(A) : J \to A \cap J,
\]
between ideals in $A \times_r G$ and invariant ideals in $A$, is bijective.
\end{defn}

\begin{rem}
The intersection map from $\Id(A \times_r G)$ to  $\Id(A)^G$ is always surjective since if $I \subseteq A$ is an invariant ideal, then $I \times_r G$ is an ideal of $A \times_r G$ with the property that $A \cap (I \times_r G) = I$.
\end{rem}

\begin{defn}
A C*-dynamical system $(A,G)$ is said to have the {\em intersection property} if $A \cap J \ne 0$ for every non-zero ideal $J$ of $A \times_r G$. It is said to have the {\em residual intersection property} if the quotient C*-dynamical system $(A/I, G)$ has the intersection property for every invariant ideal $I$ of $A$.
\end{defn}

For an invariant ideal $I$ in $A$, let $I \times_r G$ denote the ideal in $A \times_r G$ generated by $I$. The quotient map $\pi : A \to A/I$ extends to a map $\pi \times_r \operatorname{id} : A \times_r G \to A/I \times_r G$. Let $I \overline{\times}_r G = \ker \pi \times_r \operatorname{id}$. It is clear that $I \times_r G \subseteq I \overline{\times}_r G$. The C*-dynamical system $(A,G)$ is said to be {\em exact} if $I \times_r G = I \overline{\times}_r G$ for all invariant ideals $I$ in $A$. 

The following theorem is due to Sierakowski \cite{S2010}.

\begin{thm}A C*-dynamical system has the ideal separation property if and only if it is exact and has the residual intersection property.
\end{thm}

The statement of the following result combines a result of Kawabe \cite{K2017}*{Theorem 3.4} and a result of Bryder \cite{B2017}*{Theorem 3.3}. 

\begin{thm} \label{thm:intersection-property-iff-extensions-have-it}
Let $(A,G)$ be a C*-dynamical system. The following are equivalent:
\begin{enumerate}
\item $(A,G)$ has the intersection property.
\item Every C*-dynamical system $(B,G)$ satisfying
\[
(A,G) \subseteq (B,G) \subseteq (I_G(A),G)
\]
has the intersection property.
\item $(I_G(A),G)$ has the intersection property.
\end{enumerate}
\end{thm}

The same result of Kawabe asserts that a commutative C*-dynamical system $(C(X),G)$ has the intersection property if and only if its injective hull $(I_G(C(X)),G)$ is topologically free. Since proper outerness is the noncommutative analogue of topological freeness, one might suspect that a noncommutative C*-dynamical system $(A,G)$ has the intersection property if and only if its injective hull $(I_G(A),G)$ is properly outer. We will see that proper outerness of $(I_G(A),G)$ does imply the intersection property. However, the next example shows that $(A,G)$ may have the interesection property even if $(I_G(A),G)$ is not properly outer.

\begin{example}

Let $G = (\mathbb{Z}/2\mathbb{Z})^2 = \{e, u, v, uv\}$ and define an action $\alpha : G \to \Aut(M_2)$ by
\[
\alpha_e = \id, \quad \alpha_u = \Ad(U), \quad \alpha_v = \Ad(V),\ \alpha_{uv} = \Ad(U)\Ad(V),
\]
where
\[
U =
\left[\begin{matrix}
1 & 0 \\
0 & -1
\end{matrix}\right], \qquad
V = 
\left[\begin{matrix}
0 & 1 \\
1 & 0
\end{matrix}\right].
\]
The C*-dynamical system $(M_2,G,\alpha)$ is injective, however it is not properly outer. Also, $UV = -VU$, so it is not inner. However, $M_2 \times G \simeq M_4$. Since $M_4$ is simple, it follows that $(M_2,G,\alpha)$ has the intersection property. 
\end{example}

\section{Pseudo-expectations}

Let $A \subseteq B$ be an inclusion of C*-algebras. In general there is no reason to expect the existence of a conditional expectation from $B$ onto $A$. However, the injectivity of $I(A)$, where $I(A)$ denotes the injective hull of $A$, implies the existence of a unital completely positive map $\phi : B \to I(A)$ satisfying $\phi|_A = \operatorname{id}|_A$. Pitts calls the map $\phi$ a {\em pseudo-expectation} for $A$ \cite{P2012} (see also \cite{PZ2015}, \cite{Z2017a} and \cite{Z2017b}). In this section we define a notion of pseudo-expectation for a C*-dynamical system.

\begin{defn}
Let $(A,G)$ be a C*-dynamical system. An equivariant unital completely positive map $\phi : A \times_r G \to I_G(A)$ is said to be a {\em pseudo-expectation} for $(A,G,\alpha)$ if $\phi|_A = \operatorname{id}|_A$.
\end{defn}

\begin{rem}
Let $(A,G)$ be a C*-dynamical system. Then every equivariant conditional expectation from $A \times_r G$ to $A$ is a pseudo-expectation for $(A,G)$. In particular, the canonical conditional expectation $E_A : A \times_r G \to A$, being equivariant, is a pseudo-expectation, implying that the set of pseudo-expectations is non-empty. The non-emptiness of the set of pseudo-expectations also follows from the injectivity of $(I_G(A),G)$: the identity map on the trivial C*-dynamical system $(\bC,G)$ is equivariant, so by injectivity it extends to an equivariant unital completely positive map $A \times_r G \to I_G(A)$.
\end{rem}

\begin{rem}
If $(B,G)$ is an injective C*-dynamical system, then the set of pseudo-expectations for $(B,G)$ coincides with the set of equivariant conditional expectations from $B \times_r G$ to $B$.
\end{rem}

\begin{thm} \label{thm:properly-outer-implies-unique-pseudo-expectation}
Let $(A,G)$ be a C*-dynamical system. If $(I_G(A),G)$ is properly outer, then the only pseudo-expectation for $(A,G)$ is the canonical one.
\end{thm}

\begin{proof}
Suppose $(I_G(A),G,\alpha)$ is properly outer and let $\phi : A \times_r G \to I_G(A)$ be a pseudo-expectation. By the injectivity of $(I_G(A),G,\alpha)$ we can extend $\phi$ to an equivariant conditional expectation $\psi : I_G(A) \times_r G \to I_G(A)$. In particular, $I_G(A)$ belongs to the multiplicative domain of $\psi$. Thus for $s \in G$ and $b \in I_G(A)$, 
\[
\psi(\lambda_s)b = \psi(\lambda_s b) = \psi(\alpha_s(b) \lambda_s) = \alpha_s(b) \psi(\lambda_s).
\]
The proper outerness of $(I_G(A),G)$ implies that $\psi(\lambda_s) = 0$ for $s \in G \setminus \{e\}$, and it follows that $\psi$ is the canonical conditional expectation from $I_G(A) \times_r G$ to $I_G(A)$. Hence $\phi$ is the canonical conditional expectation from $A \times_r G$ to $A$.
\end{proof}

We will consider a converse to Theorem \ref{thm:properly-outer-implies-unique-pseudo-expectation} in Section \ref{sec:vanishing-obstruction}.

\begin{lem} \label{lem:equivariant-pseudo-expectation-ideal}
Let $(A,G)$ be a C*-dynamical system and let $\phi : A \times_r G \to I_G(A)$ be a pseudo-expectation for $(A,G)$. Let
\[
J = \{ x \in A \times_r G : \phi(x^*x) = 0 \}.
\]
Then $J$ is an ideal of $A \times_r G$. 
\end{lem}

\begin{proof}
It is clear that $J$ is a left ideal. To see that it is also a right ideal, first observe that $A$ belongs to the multiplicative domain of $\phi$. Thus for $x \in J$ and $a \in A$,
\[
\phi(a^*x^*xa) = a^*\phi(x^*x)a = 0,
\]
giving $xa \in J$. On the other hand, for $s \in G$, the equivariance of $\phi$ implies
\[
\phi(\lambda_s^* x^* x \lambda_s) = \alpha_{s^{-1}} \circ \phi(x^* x) = 0.
\]
giving $x \lambda_s \in J$. 
\end{proof}

The non-equivariant version of Lemma \ref{lem:equivariant-pseudo-expectation-ideal} is not true in general (see \cite{PZ2015}*{Remark 3.11}). The next result can be seen as an equivariant version of \cite{PZ2015}*{Theorem 3.5}.

\begin{thm} \label{thm:intersection-property-iff-every-equivariant-pseudo-expectation-faithful}
Let $(A,G)$ be a C*-dynamical system. Then $(A,G)$ has the intersection property if and only if every pseudo-expectation for $(A,G)$ is faithful.
\end{thm}

\begin{proof}
($\Rightarrow$)
Let $\phi : A \times_r G \to I_G(A)$ be a non-faithful pseudo-expectation for $(A,G)$. Let $J = \{x \in A \times_r G : \phi(x^*x) = 0\}$. Then by Lemma \ref{lem:equivariant-pseudo-expectation-ideal}, $J$ is an ideal of $A \times_r G$. By assumption, $J \ne 0$, but $A \cap J = 0$, so we conclude that $(A,G)$ does not have the intersection property.

($\Leftarrow$)
Suppose $(A,G)$ does not have the intersection property. Then there is a non-zero ideal $J$ of $A \times_r G$ with $A \cap J = 0$. Let $\pi : A \times_r G \to (A \times_r G) / J$ denote the quotient homomorphism. The restriction $\pi|_A$ is faithful, so by the injectivity of $(I_G(A),G)$, there is an equivariant unital completely positive map $\phi : (A \times_r G) / J \to I_G(A)$ such that $\phi \circ \pi|_A = \id|_A$. The composition $\phi \circ \pi$ is a non-faithful pseudo-expectation for $(A,G)$.
\end{proof}

Combining Theorem \ref{thm:properly-outer-implies-unique-pseudo-expectation} and Theorem \ref{thm:intersection-property-iff-every-equivariant-pseudo-expectation-faithful} yields a simple proof of the following result of Sierakowski \cite{S2010}*{Remark 2.23}. 

\begin{cor} \label{cor:properly-outer-implies-intersection-property}
Let $(A,G)$ be a C*-dynamical system. If $(A,G)$ is properly outer, then it has the intersection property.
\end{cor}

In recent work, Zarikian \cite{Z2017b} independently observes that pseudo-expectations, defined in the sense of Pitts, can be used to prove Corollary~\ref{cor:properly-outer-implies-intersection-property}.

\section{Noncommutative boundaries}

\subsection{Matrix convexity}

We require some basic notions from the theory of matrix convexity. For a locally convex topological vector space $V$ and $n \geq 1$, let $M_n(V)$ denote the space of $n \times n$ matrices over $V$. A {\em matrix convex set} is a family $\bK = (K_n)_{n \geq 1}$ of convex subsets $K_n \subseteq M_n(V)$ closed under taking {\em matrix convex combinations}, i.e. such that 
\[
\sum_{i=1}^k \gamma_i^* \mu_i \gamma_i \in K_n
\]
for all $\mu_i \in K_{n_i}$ and $\gamma_i \in M_{n_i,n}$ for $1 \leq i \leq k$ satisfying $\sum_{i=1}^k \gamma_i^* \gamma_i = 1$. If $K_n$ is compact in the product topology on $M_n(V)$ for each $n \geq 1$, then $\bK$ is said to be {\em compact}.

Let $\bK$ and $\bL$ be compact matrix convex sets. A {\em matrix affine map} is a family $f = (f_n)_{n \geq 1}$ of mappings $f_n : K_n \to L_n$ satisfying
\[
f_n \left( \sum_{i=1}^k \gamma_i^* \mu_i \gamma_i \right) = \sum_{i=1}^k \gamma_i^* f_{n_i}(\mu_i) \gamma_i,
\]
for all $\mu_i \in K_{n_i}$ and $\gamma_i \in M_{n_i,n}$ for $1 \leq i \leq k$ satisfying $\sum_{i=1}^k \gamma_i^* \gamma_i = 1$. We write $A(\bK)$ for the set of all continuous matrix affine mappings from $\bK$ to $(M_n(\bC))$. A matrix affine $f = (f_n)$ from $\bK$ to $(M_n(\bC))$ is continuous if $f_1$ is continuous (see the remarks following \cite{WW1999}*{Definition 3.4}). 

A {\em homeomorphism} of $\bK$ is a family $g= (g_n)_{n \geq 1}$ of homeomorphisms $g_n : K_n \to K_n$. We will write $\operatorname{Homeo}(\bK)$ for the set of homeomorphisms of $\bK$.

For an operator system $S$, let $MS(S)$ denote the matrix state space of $S$, i.e. $MS(S) = (MS_n(S))_{n \geq 1}$, where $MS_n(S)$ denotes the set of unital completely positive maps from $S$ into $M_n$. Then $MS(S)$ is a compact matrix convex set with respect to the weak* topology on $S^*$. Note that each $a \in S$ gives rise to a matrix affine function from $MS(S)$ to $(M_n(\bC))$. 

The first part of a result of Webster and Winkler \cite{WW1999}*{Proposition 3.5} provides an analogue for general operator systems of Kadison's representation theorem for operator subsystems of commutative C*-algebras.

\begin{thm}
Let $S$ be an operator system with matrix state space $MS(S)$ and let $A(MS(S))$ denote the operator system of matrix affine functions on $MS(S)$. Then $S$ and $A(MS(S))$ are isomorphic.
\end{thm}

The next result is equivalent to the second part of Webster and Winkler's result \cite{WW1999}*{Theorem 3.5}. 

\begin{lem} \label{lem:adjoint-complete-isometry-iff-homeomorphism}
Let $\bK$ and $\bL$ be compact matrix convex sets, let $\phi : A(\bK) \to A(\bL)$ be a unital completely positive map and let $f : \bL \to \bK$ denote the corresponding induced map. Then $\phi$ is completely isometric if and only if $f$ is a matrix affine homeomorphism. 
\end{lem}

\subsection{Matrix affine dynamical systems}

A {\em matrix affine dynamical system} is a triple $(\bK,G,h)$ consisting of a compact matrix convex space $\bK = (K_n)_{n \geq 1}$, a discrete group $G$ and an action of $G$ on $\bK$ in the form of a group homomorphism $h : G \to \operatorname{Homeo}(\bK)$ such that $h_s$ is a matrix affine homeomorphism for each $s \in G$.

Let $(S,G,\alpha)$ be a noncommutative dynamical system. By the representation theorem of Webster and Winkler, we can identify $S$ with the operator system $A(MS(S))$ of matrix affine functions on $MS(S)$. We obtain a corresponding matrix affine dynamical system $(MS(S),G,h)$, where $h$ denotes the action on $MS(S)$ induced by $\alpha$. For $n \geq 1$ and $\mu \in MS_n(S)$, $h_s(\mu) = \mu \circ \alpha_{s^{-1}}$.

\subsection{Boundaries for matrix affine dynamical systems}

In this section we obtain a duality result for noncommutative C*-dynamical systems that is analogous to Theorem \ref{thm:affine-boundary-iff-essential-extension}. Motivated by the results in the previous section, along with Webster and Winkler's representation theorem \cite{WW1999}*{Theorem 3.5}, we work within the framework of matrix convexity.

\begin{defn} \label{defn:matrix-affine-boundary}
Let $(\bK,G)$ be a matrix affine dynamical system. We say that a pair $((\bL,G),f)$ consisting of a matrix affine dynamical system $(\bL,G)$ and a continuous surjective matrix affine map $f : \bL \to \bK$ is a {\em boundary} for $(\bK,G)$ if whenever $\bL' \subseteq \bL$ is an invariant compact matrix convex subset satisfying $f(\bL') = \bK$, then $\bL' = \bL$.
\end{defn}

\begin{thm} \label{thm:nc-boundary-iff-essential-extension}
Let $(\bK,G)$ and $(\bL,G)$ be matrix affine dynamical systems and let $f : \bL \to \bK$ be an equivariant continuous surjective matrix affine map. Then $((\bL,G),f)$ is a boundary for $(\bK,G)$ if and only if the corresponding extension $((A(\bL),G),\iota)$ of $(A(\bK),G)$ is essential, where $\iota : A(\bK) \to A(\bL)$ denotes the equivariant embedding defined by $\iota(a)(\mu) = a(f(\mu))$ for $a \in A(\bK)$ and $\mu \in \bL$.
\end{thm}

\begin{proof}
($\Rightarrow$)
Suppose $(A(\bL),G),\iota)$ is an essential extension of $(A(\bK),G)$ and let $\bL' \subseteq \bL$ be an invariant compact matrix convex subset satisfying $f(\bL') = \bK$. Let $\rho : A(\bL) \to A(\bL')$ denote the restriction map. Then $\rho \circ \iota$ is completely isometric, so by essentiality, $\rho$ is completely isometric. The restriction of the adjoint map of $\rho$ to $\bL'$ is the inclusion map from $\bL'$ into $\bL$. Hence by Lemma \ref{lem:adjoint-complete-isometry-iff-homeomorphism}, $\bL' = \bL$.

($\Leftarrow$)
Suppose $((\bL,G),f)$ is a boundary for $(\bK,G)$. Let $(\bN,G)$ be a matrix affine dynamical system and let $\phi : A(\bL) \to A(\bN)$ be an equivariant unital completely positive map such that $\phi \circ \iota$ is completely isometric. Let $g : \bN \to \bL$ denote the restriction of the adjoint of $\phi$ and let $\bL' = g(\bN)$. Since $\phi \circ \iota$ is completely isometric, it follows from Lemma \ref{lem:adjoint-complete-isometry-iff-homeomorphism} that $f \circ g$ is a matrix affine homeomorphism. In particular, $f(\bL') = \bK$. Since $((\bL,G),f)$ is a boundary for $(\bK,G)$, it follows that $\bL' = \bL$. Hence by Lemma \ref{lem:adjoint-complete-isometry-iff-homeomorphism}, $\phi$ is completely isometric.
\end{proof}

The next result follows immediately from Definition \ref{defn:matrix-affine-boundary} and Theorem \ref{thm:nc-boundary-iff-essential-extension}.

\begin{cor}
Let $(A,G)$ and $(B,G)$ be C*-dynamical systems. Let $\iota : A \to B$ be an equivariant embedding and let $f : MS(B) \to MS(A)$ denote the corresponding induced map. Then $((B,G),\iota)$ is an essential extension of $(A,G)$ if and only if whenever $\bK \subseteq MS(B)$ is an invariant closed matrix convex subset satisfying $f(\bK) = MS(A)$, then $\bK = MS(B)$.
\end{cor}

Let $(\bK,G)$ be a matrix affine dynamical system and let $(A(\bK),G)$ denote the corresponding noncommutative dynamical system with injective hull $(I_G(A(\bK)),G)$. Let $\tilde{\bK} = MS(I_G(A(\bK)))$. Then by Webster and Winkler's representation theorem, $(I_G(A(\bK)),G)$ is isomorphic to $(A(\tilde{\bK}),G)$. The next result shows that the corresponding matrix affine dynamical system $(\tilde{\bK},G)$ is a kind of ``projective cover'' of $(\bK,G)$. 

\begin{prop} \label{prop:projectivity-dual-injectivity}
Let $(\bK,G)$ be a matrix affine dynamical system with corresponding noncommutative dynamical system $(A(\bK),G)$. Let $(A(\tilde{\bK}),G)$ denote the injective hull of $(A(\bK),G)$ with corresponding matrix affine dynamical system $(\tilde{\bK},G)$ and let $f : \tilde{\bK} \to \bK$ denote the map induced by the inclusion map $(A(\bK),G) \subseteq (A(\tilde{\bK}),G)$. Let $(\bL,G)$ be a matrix affine dynamical system and let $g : \bL \to \bK$ be an equivariant continuous surjective matrix affine map. Then there is an equivariant continuous matrix affine map $h : \tilde{\bK} \to \bL$ such that $f = g \circ h$. Moreover, $((h(\tilde{\bK}),G),g|_{h(\tilde{\bK})})$ is a boundary for $(\bK,G)$.

\begin{equation*}
\begin{tikzcd}
\tilde{\bK} \arrow{d}{h} \arrow{dr}{f} \\
\bL \arrow{r}{g} & \bK
\end{tikzcd}
\end{equation*}
\end{prop}

\begin{proof}
We obtain inclusions $A(\bK) \subseteq A(\bL)$ and $A(\bK) \subseteq A(\tilde{\bK})$. By injectivity there is an equivariant unital completely positive map $\phi : A(\bL) \to A(\tilde{\bK})$ such that $\phi|_{A(\bK)} = \operatorname{id}|_{A(\bK)}$. Let $h : \tilde{\bK} \to \bL$ denote the restriction of the adjoint of $\phi$. Then $f = g \circ h$. The fact that $((h(\tilde{\bK}),G),g|_{h(\tilde{\bK})})$ is a boundary follows from the inclusion of noncommutative dynamical systems $(A(\bK),G) \subseteq (\phi(A(\bL)),G) \subseteq (A(\tilde{\bK}),G)$.
\end{proof}

The next result follows immediately from Proposition \ref{prop:projectivity-dual-injectivity}.

\begin{cor} \label{cor:contains-boundary}
Let $(\bK,G)$ and $(\bL,G)$ be matrix affine dynamical systems and let $g : \bL \to \bK$ be an equivariant continuous surjective matrix affine map. There is an invariant closed matrix convex subset $\bL' \subseteq \bL$ such that $((\bL',G), g|_{\bL'})$ is a boundary for $(\bK,G)$.
\end{cor}

\subsection{Pseudo-expectations and noncommutative boundaries}

Let $(A,G)$ be a C*-dynamical system and let $E_A : A \times_r G \to A$ denote the canonical conditional expectation. Consider the invariant closed matrix convex subset $\bK_A = \{ \mu \circ E_A : \mu \in MS(A) \} \subseteq MS(A \times_r G)$. Let $\rho : \bK_A \to MS(A)$ denote the restriction map. Then $\rho$ is an equivariant continuous surjective matrix affine map and $((\bK_A,G),\rho)$ is a boundary for $(MS(A),G)$ in the sense of Definition \ref{defn:matrix-affine-boundary}. We say that $(\bK_A,G)$ is the {\em canonical boundary for $(MS(A),G)$ in $MS(A \times_r G)$}.

\begin{thm} \label{thm:nc-boundary-correspondence-pseudo-expectation}
Let $(A,G)$ be a C*-dynamical system. There is a bijective correspondence between boundaries for $(MS(A),G)$  in $(MS(A \times_r G),G)$ and pseudo-expectations for $(A,G)$.
\end{thm}

\begin{proof}
($\Rightarrow$)
Let $((\bK,G),f) \subseteq MS(A \times_r G)$ be a boundary for $(MS(A),G)$. Then by Theorem \ref{thm:nc-boundary-iff-essential-extension}, $((A(\bK),G),\iota)$ is an essential extension of $(A,G)$, where $\iota : A \to A(\bK)$ denotes the equivariant embedding defined by $\iota(a)(y) = a(f(y))$ for $a \in A$ and $y \in \bK$. Hence there is an extension $((I_G(A),G),\iota)$ of $(A(\bK),G)$. Let $\phi : A \times_r G \to I_G(A)$ denote composition of $\iota$ with the restriction map to $A(\bK)$. Then $\phi$ is an equivariant unital completely positive map and $\phi|_A = \id|_A$. In particular, $\phi$ is a pseudo-expectation for $(A,G)$.

($\Leftarrow$)
Conversely, let $\phi : A \times_r G \to I_G(A)$ be a pseudo-expectation for $(A,G)$. Let $\bK = MS(\phi(A \times_r G))$. Then Webster and Winkler's representation theorem implies that we can identify $\phi(A \times_r G)$ with $A(\bK)$. Let $f : \bK \to MS(A \times_r G)$ denote the restriction of the adjoint of $\phi$ and let $g : MS(A \times_r G) \to MS(A)$ denote the restriction map. Then $((f(\bK), G),g)$ is a boundary for $(MS(A),G)$. 
\end{proof}

\section{Vanishing obstruction and partial representations}

In this section we identify a connection with the theory of partial C*-dynamical systems. For a reference, we refer the reader to the book of Exel \cite{E2017}.

\begin{defn}\label{defn:partial-dynamical-system}
Let $A$ be a unital C*-algebra and let $G$ be a discrete group. A {\em partial action} of $G$ on $A$ is a pair $(\{A_t\}_{t \in G}, \{\alpha_t\}_{t \in G})$ consisting of a family of ideals $\{A_t\}_{t \in G}$ in $A$ and a family of isomorphisms $\{\alpha_t \}_{t \in G}$ with $\alpha_t : A_{t^{-1}} \to A_t$ satisfying
\begin{enumerate}
\item $A_e = A$,
\item $\alpha_s(A_{s^{-1}} \cap A_t) = A_s \cap A_{st}$ for $s,t \in G$ and
\item $\alpha_s \circ \alpha_t = \alpha_{st}$ on $A_{t^{-1}} \cap A_{t^{-1}s^{-1}}$ for $s,t \in G$.
\end{enumerate}
The tuple $(A,G,\{A_t\}_{t \in G}, \{\alpha_t\}_{t \in G})$ is said to be a {\em partial C*-dynamical system}. 
\end{defn}

Let $(B,G,\beta)$ be a C*-dynamical system such that $B$ is injective. For $s \in G$ let $p_s \in B$ denote the largest $\beta_s$-invariant projection such that $\beta_s|_{Bp_s}$ is inner and let $u_s \in Bp_s$ be a unitary such that $\beta_s|_{Bp_s} = \Ad(u_s)$ and $u_e = 1$.  Let $B_s = Bp_s$. Then the tuple $(B, G, \{B_s\}_{s \in G}, \{\Ad(u_s)\}_{s \in G})$ is a partial C*-dynamical system in the sense of Definition \ref{defn:partial-dynamical-system} that we will refer to as the partial C*-dynamical system corresponding to $(B,G,\beta)$. This definition does not depend on the choice of $u$. It is easy to check that for $s,t \in G$, the projections $p_s$ and $p_t$ satisfy $p_s = p_{s^{-1}}$ and $p_s p_t \leq p_{st}$. Furthermore, by Theorem \ref{thm:prop-outer-implies-ess-prop-outer}, $\beta_t(p_s) = p_{tst^{-1}}$. 

\begin{defn} \label{defn:partial-representation}
Let $A$ be a C*-algebra and let $G$ be a discrete group. A {\em partial *-representation} of $G$ in $A$ is a partial isometry-valued map $u : G \to A$ satisfying
\begin{enumerate}
\item $u_e = 1$,
\item $u_s u_t u_{t^{-1}} = u_{st} u_{t^{-1}}$ for $s,t \in G$ and
\item $u_{t^{-1}} = (u_t)^*$ for $t \in G$.
\end{enumerate}
\end{defn}

For a C*-dynamical system $(B,G,\beta)$ such that $B$ is injective, let $(B, G, \{B_s\}_{s \in G}, \{\Ad(u_s)\}_{s \in G})$ denote the corresponding partial dynamical system. It is clear that the map $u : G \to B$ satisfies properties (1) and (3) in Definition \ref{defn:partial-representation}. It follows from \cite{E2017}*{Proposition 9.6} that $u$ is a partial *-representation if and only if $u_s u_t p_s p_t = u_{st}$.

\begin{defn}[Vanishing obstruction]
Let $(A,G)$ be a C*-dynamical system with extension $(I(A),G)$ and injective hull $(I_G(A),G)$. We will say that $(A,G,\alpha)$ has {\em vanishing obstruction} if the map $u : G \to I(A)$ defined as in Theorem \ref{thm:prop-outer-implies-ess-prop-outer} can be chosen to be a partial *-representation of $G$ satisfying $I(\alpha)_t(u_s) = u_{tst^{-1}}$ for all $s,t \in G$.
\end{defn}

\begin{prop} \label{prop:vanishing-obstruction-partial-representation}
Let $(A,G)$ be a C*-dynamical system with vanishing obstruction. For $s \in G$ let $p_s \in I(A)$ denote the largest $\alpha_s$-invariant projection in $I(A)$ such that $I(\alpha)_s|_{I(A)p_s}$ is inner. Then there is a partial *-representation $u : G \to I(A)$ satisfying
\begin{enumerate}
\item $u_s^* u_s = p_s = u_s u_s^*$ for $s \in G$,
\item $I(\alpha)_s|_{I(A)p_s} = \Ad(u_s)$ for $s \in G$,
\item $I(\alpha)_s(u_t) = u_{sts^{-1}}$ for $s,t \in G$.
\end{enumerate}
\end{prop}

\begin{example} \label{ex:partial-cohomology-1}
Let $(A,G)$ be a  C*-dynamical system with $A$ commutative. Then the extension $(I(A),G)$ is also commutative. For $s \in G$, let $p_s \in I(A)$ denote the largest $I(\alpha)_s$-invariant projection such that $I(\alpha)_s|_{I(A)p_s}$ is inner. Since $I(A)$ is commutative, this is equivalent to the condition that $I(\alpha)_s|_{I(A)p_s}$ is trivial. By Theorem \ref{thm:prop-outer-implies-ess-prop-outer}, for $t \in G$, $I(\alpha)_s(p_t) = p_{sts^{-1}}$, so $(A,G)$ has vanishing obstruction.
\end{example}

\begin{example} \label{ex:partial-cohomology-2}
Let $(A,G,\alpha)$ be a C*-dynamical system with the property that for every $s \in G$, either $\alpha_s$ is properly outer or quasi-inner. Let $G_A = \{s \in G : \alpha_s \text{ is quasi-inner}\}$ and suppose that the cohomology group $H^2(G_A,UZ(I(A)))$ is trivial, where $UZ(I(A))$ denotes the unitary group of the center of $I(A)$. Then the $2$-cocycle $\sigma : G_A \times G_A \to UZ(I(A))$ defined by $\sigma(s,t) = u_s u_t u_{st}^* $ is a $2$-coboundary, so the restriction $u|_{G_A}$ can be chosen to be a representation of $G_A$. It is clear in this case that $u$ is a partial *-representation of $G$. Hence $(G,A)$ has vanishing obstruction. 

This applies in particular if $A$ is prime, since \cite{H1981}*{Theorem 7.1} and \cite{H1982b} imply that $I(A)$ has trivial center. In this case, the partial cohomology group $H^2(G_A,UZ(I(A)))$ reduces to the Mackey obstruction $H^2(G_A,\mathbb{T})$. It is well known that $H^2(G_A,\mathbb{T}) = 0$ if either $G_A$ is a free group (hence in particular if $G$ is free by the Nielsen-Schreier theorem) or if $G_A$ is finite and every Sylow subgroup is cyclic (hence in particular if $G_A$ is cyclic).
\end{example}

The condition of vanishing obstruction is related to the theory of partial cohomology for groups introduced by Dokuchaev and Khrypchenko \cite{DK2015}.

\begin{defn}
Let $(A,G,\{A_t\}_{t \in G}, \{\alpha_t\}_{t \in G})$ be a unital partial C*-dynamical system. A {\em partial $1$-cochain} is a function $f : G \to A$ such that $f(t) \in A_t^{-1}$ for all $t \in G$. A {\em partial $2$-cochain} is a function $\sigma : G \times G \to A$ such that $\sigma(s,t) \in (A_s \cap A_{st})^{-1}$ for all $s,t \in G$. A partial $2$-cochain $\sigma$ is a {\em partial $2$-cocycle} if it satisfies 
\[
\alpha_r(p_{r^{-1}} \sigma(s,t)\sigma(r,st))\sigma(rs,t) = \sigma(rs,t) \sigma(r,s)
\]
for all $r,s,t \in G$, where $p_r$ denotes the unit in $A_r$. We let $Z_p^2(G,A)$ denote the set of all partial $2$-cocycles.

A partial $2$-cocycle $\sigma$ is a {\em partial $2$-coboundary} if there is a partial 1-cochain $f : G \to A$ such that
\[
\sigma(s,t) = \alpha_s(p_{s^{-1}}f(t))f(st)^{-1}f(s)
\]
for all $s,t \in G$. We let $B_p^2(G,A)$ denote the set of all partial $2$-coboundaries.

The sets $Z_p^2(G,A)$ and $B_p^2(G,A)$ are abelian groups. The {\em partial 2-cohomology group} $H_p^2(G,A)$ is the quotient
\[
H_p^2(G,A) = Z_p^2(G,A)/B_p^2(G,A).
\]
We will refer to $H_p^2(G,A)$ as the {\em second partial cohomology group for $G$ with coefficients in $A$}.
\end{defn}

\begin{thm} \label{thm:the-partial-cocycle}
Let $(B,G,\beta)$ be a C*-dynamical system such that $B$ is injective and let $(B, \{B_s\}_{s \in G}, \{\Ad(u_s)\}_{s \in G})$ denote the corresponding partial dynamical system. Define $\sigma : G \times G \to B$ by
\[
\sigma(s,t) = u_s u_t u_{st}^*,
\]
for $s,t \in G$. Then $\sigma$ is a $Z(B)$-valued partial $2$-cocycle. If the map $u : G \to B$ can be chosen to be a partial *-representation of $G$, then $\sigma$ is a partial $2$-coboundary.
\end{thm}

\begin{proof}
It is clear that $\sigma(s,t) \in Z(B)$ for all $s,t \in G$, and from this it is not difficult to check that $\sigma$ is a partial $2$-cocycle. If $u$ is a partial *-representation, then it is clear that $\sigma$ is a partial $2$-coboundary. 
\end{proof}

In general, the second partial cohomology group can be much more complicated than the second cohomology group. For example, it can be non-trivial even for free groups (see e.g. the survey \cite{P2015}).

\section{The intersection property} \label{sec:vanishing-obstruction}

\subsection{Proper outerness}

\begin{lem} \label{lem:canonical-non-canonical-conditional-expectation}
Let $(A,G,\alpha)$ be a C*-dynamical system with vanishing obstruction and let $u : G \to I(A)$ be a partial *-representation as in Proposition \ref{prop:vanishing-obstruction-partial-representation}. For $s \in G$, let $q_s \in I_G(A)$ denote the largest $I_G(\alpha)_s$-invariant projection such that $I_G(\alpha)_s|_{I_G(A)q_s}$ is inner. Then there is an equivariant conditional expectation $\phi : I_G(A) \times_r G \to I_G(A)$ satisfying $\phi(b \lambda_s) = b u_s q_s$ for $b \in I_G(A)$ and $s \in G$.
\end{lem}

\begin{proof}
For $z \in \Glimm(I_G(A))$, let $r_z \in I_G(A)^{**}$ denote the central open projection corresponding to $z$, so that $z = I_G(A)^{**}r_z$ and $(I_G(A)/z)^{**} \simeq I_G(A)^{**}(1-r_z)$. Let $G_z \leq G$ denote the quasi-stabilizer subgroup corresponding to $z$, so that $G_z = \{ s \in G : z \in \operatorname{supp}(q_s) \}$.

Let $\sigma_z : I(A) \to I_G(A)^{**}(1-r_z)$ denote the compression map. The subgroup $G_z$ is amenable by Proposition \ref{prop:quasi-stabilizer-for-injective-is-amenable}, and $(\sigma_z,\sigma_z \circ u)$ is a covariant pair for the C*-dynamical system $(I(A),G_z)$.  Therefore, we obtain a representation $\rho_z : I(A) \times G_z \to I_G(A)^{**}(1-r_z)$.

Let $E_{G_z} : I(A) \times_r G \to I(A) \times_r G_z$ denote the conditional expectation satisfying
\[
E_{G_z}(b \lambda_t) =
\begin{cases}
b \lambda_t & \text{if } t \in G_z, \\
0 & \text{otherwise},
\end{cases}
\]
for $b \in I(A)$ and $t \in G$. Define $\phi_z : I(A) \times_r G \to I_G(A)^{**}(1-r_z)$ by $\phi_z = \rho_z \circ E_{G_z}$. Then
\[
\phi_z(b\lambda_t) =
\begin{cases}
b u_t (1-r_z) & \text{if } t \in G_z, \\
0 & \text{otherwise},
\end{cases}
\]
for $b \in I(A)$ and $t \in G$.

Representations of $I_G(A)$ with kernels that are distinct Glimm ideals are disjoint. Hence the projections $\{1 - r_z : z \in \Glimm(I_G(A)) \}$ are mutually orthogonal. Define a unital completely positive map $\theta : I(A) \times_r G \to \oplus_{z \in \Glimm(I_G(A))} I_G(A)^{**}(1-r_z)$ by $\theta = \bigoplus_{z \in \Glimm(I_G(A))} \phi_z$.

The fact that $\cap_{z \in \Glimm(I_G(A))} z = 0$ implies there is a unital completely positive map $\psi : \bigoplus_{z \in \Glimm(I_G(A))} I_G(A)^{**}(1-r_z) \to I_G(A)$ such that the composition $\phi = \psi \circ \theta : I(A) \times G_z \to $ satisfies $\phi'(b \lambda_t) = b u_t q_t$ for $b \in I(A)$ and $t \in G$.

Since $(A,G)$ has vanishing obstruction, Theorem \ref{thm:prop-outer-implies-ess-prop-outer} implies $\phi'$ is an equivariant unital completely positive map. By the injectivity of $(I_G(A),G)$, we may extend it to an equivariant unital completely positive map $\phi : I_G(A) \times_r G \to I_G(A)$.
\end{proof}

\begin{rem}
It is necessary to work with $I(A)$ in the proof of Lemma \ref{lem:canonical-non-canonical-conditional-expectation} because the assumption that $(A,G)$ has vanishing obstruction applies to $I(A)$, and not necessarily to $I_G(A)$. 
\end{rem}

\begin{thm} \label{thm:vanishing-obstruction-intersection-property-iff-properly-outer}
Let $(A,G)$ be a C*-dynamical system. Then $(A,G)$ has the intersection property if its injective hull $(I_G(A),G)$ is properly outer. If $(A,G)$ has vanishing obstruction, then the converse is also true.
\end{thm}

\begin{proof}
($\Leftarrow$)
If $(I_G(A),G)$ is properly outer, then $(I_G(A),G)$ has the intersection property by Theorem \ref{cor:properly-outer-implies-intersection-property}. It follows from Theorem \ref{thm:intersection-property-iff-extensions-have-it} that $(A,G)$ has the intersection property.

($\Rightarrow$)
Suppose $(A,G)$ has vanishing obstruction and $(I_G(A),G)$ is not properly outer. Let $u : G \to I(A)$ be a map as in Proposition \ref{prop:vanishing-obstruction-partial-representation}. For $s \in G$, let $q_s \in I_G(A)$ denote the largest $I_G(\alpha)_s$-invariant projection such that $I_G(\alpha)_s|_{I_G(A)q_s}$ is inner. By assumption there is $s \in G$ such that $0 \ne q_s$. By Theorem \ref{thm:prop-outer-implies-ess-prop-outer}, $q_s \leq p_s = u_s u_s^*$, so $u_s \ne 0$. 

Let $\phi : I_G(A) \times_r G \to I_G(A)$ be the equivariant conditional expectation as in Lemma \ref{lem:canonical-non-canonical-conditional-expectation}. Then
\begin{align*}
\phi(q_s (u_s - \lambda_s)^*(u_s - \lambda_s)) &= q_s \phi(u_s^* u_s - u_s^* \lambda_s - \lambda_s^* u_s + \lambda_s^* \lambda_s) \\
&= q_s - u_s^* \phi(\lambda_s) - \phi(\lambda_s^*) u_s + q_s \\
&= q_s - u_s^*u_s q_s - q_s u_s^* u_s + q_s \\
&= 0.
\end{align*}
Hence $\phi$ is not faithful, and it follows from Theorem \ref{thm:intersection-property-iff-every-equivariant-pseudo-expectation-faithful} that $(I_G(A),G)$ does not have the intersection property. Therefore, by Theorem \ref{thm:intersection-property-iff-extensions-have-it}, $(A,G)$ does not have the intersection property.
\end{proof}

\begin{cor}
Let $(A,G)$ be a C*-dynamical system such that $A$ is prime and let $G_A = \{s \in G : \alpha_s \text{ is quasi-inner} \}$. If $G_A$ has vanishing Mackey obstruction $H^2(G_A,\mathbb{T})$, then $(A,G)$ has the intersection property if and only if its injective hull $(I_G(A),G)$ is properly outer. 
\end{cor}

\begin{proof}
If $A$ is prime, then it follows as in Example \ref{ex:partial-cohomology-2} that $(A,G)$ has vanishing obstruction if and only if the Mackey obstruction $H^2(G_A,\mathbb{T})$ vanishes.
\end{proof}

Let $(A,G)$ be a C*-dynamical system with injective hull $(I_G(A),G)$. A key part of the proof of Lemma \ref{lem:canonical-non-canonical-conditional-expectation} is the amenability of quasi-stabilizers of points in $\Glimm(I_G(A))$.  However, if $G$ is amenable, then quasi-stabilizers of points in $\Glimm(I(A))$ are amenable, so we can instead construct the conditional expectation from $I(A) \times_r G$ to $I(A)$. By Theorem \ref{thm:inj-env-characterization-quasi-inner-properly-outer}, $(I(A),G)$ is properly outer if and only if $(A,G)$ is properly outer. Thus, arguing exactly as in the proof of Theorem \ref{thm:vanishing-obstruction-intersection-property-iff-properly-outer} yields the following results. Note that the forward direction of the next result is Corollary \ref{cor:properly-outer-implies-intersection-property}. 

\begin{thm} \label{thm:amenable-vanishing-obstruction-intersection-property-iff-properly-outer}
Let $(A,G)$ be a C*-dynamical system such that $G$ is amenable. If $(A,G)$ is properly outer, then it has the intersection property. If $(A,G)$ has vanishing obstruction, then the converse is also true. 
\end{thm}

\begin{cor}
Let $(A,G)$ be a C*-dynamical system such that $A$ is prime and $G$ is amenable and let $G_A = \{s \in G : \alpha_s \text{ is quasi-inner} \}$. If $G_A$ has vanishing Mackey obstruction $H^2(G_A,\mathbb{T})$, then $(A,G)$ has the intersection property if and only if it is properly outer.
\end{cor}

If $A$ is commutative, then Example \ref{ex:partial-cohomology-1} shows that $(A,G)$ has vanishing obstruction. Since proper outerness is equivalent to topological freeness for commutative C*-dynamical systems, Theorem \ref{thm:amenable-vanishing-obstruction-intersection-property-iff-properly-outer} is a significant generalization of a result of Kawamura and Tomiyama \cite{KT1990}*{Theorem 4.1} (see also \cite{AS1994}*{Theorem 2}). The next result generalizes \cite{KT1990}*{Theorem 4.4} (see also \cite{AS1994}*{Corollary 3}.

\begin{cor}
Let $(A,G)$ be a minimal C*-dynamical system such that $G$ is amenable. If $(A,G)$ is properly outer, then the reduced crossed product $A \times_r G$ is simple. If $(A,G)$ has vanishing obstruction, then the converse is also true.  
\end{cor}

The role of the Mackey obstruction in the next result follows as in Example \ref{ex:partial-cohomology-2}.

\begin{cor}
Let $(A,G)$ be a minimal C*-dynamical system such that $A$ is prime and $G$ is amenable and let $G_A = \{s \in G : \alpha_s \text{ is quasi-inner} \}$. If $G_A$ has vanishing Mackey obstruction $H^2(G_A,\mathbb{T})$, then the reduced crossed product $A \times_r G$ is simple if and only if $(A,G)$ is properly outer.
\end{cor}

\subsection{Intrinsic characterization in the general case} \label{sec:intrinsic-general}

Let $A$ be a unital C*-algebra. There is a bijective correspondence between ideals in $A$ and subsets of $\Prim(A)$ that are closed in the hull-kernel topology (see e.g. \cite{P1979}*{Theorem 4.1.3}). For an ideal $I$ in $A$, we will let $\hull(I) \subseteq \Prim(A)$ denote the corresponding closed subset, so that
\[
\operatorname{hull}(I) = \{P \in \Prim(A) : P \supseteq I\}.
\]
Then $\hull(I)$ can be identified with $\Prim(A/I)$. We will say that a set $\{I_i\}$ of ideals in $A$ {\em covers} $A$ if $\cup \operatorname{hull}(I_i) = \Prim(A)$.

Let $\Sub(G)$ denote the set of subgroups of $G$ and let $\Sub_a(G) \subseteq \Sub(G)$ denote the subset of amenable subgroups of $G$. The Chabauty topology on $\Sub(G)$ coincides with the relative topology when $\Sub(G)$ is viewed as a subset of $\{0,1\}^G$ equipped with the product topology. For a net $(H_i)$ in $\Sub(G)$ and $H \in \Sub(G)$,  $\lim H_i = H$ if
\begin{enumerate}
\item every $h \in H$ eventually belongs to $H_i$ and
\item for every subnet $(H_j)$ of $(H_i)$, $\cap_j H_j \subseteq H$.
\end{enumerate}
We consider the compact dynamical system $(\Sub(G),G)$, where $\Sub(G)$ is equipped with the Chabauty topology and the action is by conjugation. Since $\Sub_a(G)$ is closed and invariant, we obtain an inclusion of compact dynamical systems $(\Sub_a(G),G) \subseteq (\Sub(G),G)$.

More generally, we consider the compact dynamical system $(\Glimm_p(A) \times \Sub(G),G)$, where $\Glimm_p(A) \times \Sub(G)$ is equipped with the product topology.

\begin{defn} \label{defn:uniformly-recurrent-subset}
Let $(A,G)$ be a C*-dynamical system. We will say that an invariant closed subset $X \subseteq \Glimm_p(A) \times \Sub(G)$ {\em covers} $(A,G)$ or is {\em $(A,G)$-covering} if it has the following two properties:
\begin{enumerate}
\item The ideals in $\pr_1(X)$ cover $A$, where $\pr_1 : \Glimm_p(A) \times \Sub(G) \to \Glimm_p(A)$ denotes the projection onto the first coordinate.
\item For each $(I,H) \in X$, $H \leq G_I$, where $G_I$ denotes the quasi-stabilizer subgroup corresponding to $I$.
\end{enumerate}
We will say that an $(A,G)$-covering subset $X \subseteq \Glimm_p(A) \times \Sub(G)$ is {\em amenable} if $\pr_2(X) \subseteq \Sub_a(G)$, where  $\pr_2 : \Glimm_p(A) \times \Sub(G) \to \Sub(G)$ denotes the projection onto the second coordinate. If $\pr_2(X) = \{\{e\}\}$, then we will say that $X$ is {\em trivial}.
\end{defn}

\begin{rem}
If $A = \bC$, then minimal $(A,G)$-covering subsets of $\Glimm_p(A) \times \Sub(G)$ correspond to uniformly recurrent subgroups of $G$ as defined by Glasner and Weiss \cite{GW2015}.
\end{rem}

\begin{lem} \label{lem:intermediate-map}
Let $(A,G,\alpha)$ be a C*-dynamical system with vanishing obstruction and let $u : G \to I(A)$ be a partial *-representation of $G$ as in Proposition \ref{prop:vanishing-obstruction-partial-representation}. For a pseudo-Glimm ideal $I \in \Glimm_p(A)$ and an amenable subgroup $H \leq G_I$, there is a unital completely positive map $\phi_{(I,H)} : I(A) \times_r G \to \B(H_{(I,H)})$ with the following properties:
\begin{enumerate}
\item The restriction $\phi_{(I,H)}|_A$ is equivalent to the composition of the quotient map from $A$ onto $A/I$ with the universal representation of $A/I$ and
\item $\phi_{(I,H)}(u_s^* \lambda_s) = \chi_H(s)1_{H_{(I,H)}}$ for $s \in G$, where $\chi_H$ denotes the indicator function for $H$.
\end{enumerate}
\end{lem}

\begin{proof}
Let $J$ be a proper primal ideal in $I(A)$ such that $A \cap J = I$. Then for $s \in G_I$, $1-p_s \in J$, so for $b \in I(A)$, 
\[
I(\alpha)_s(b) - u_s b u_s^* = I(\alpha)_s(b(1 - p_s)) \in J.
\]
Let $\sigma : I(A) \to \B(H_u)$ denote the composition of the quotient map from $I(A)$ to $I(A)/J$ with the universal representation of $I(A)/J$. For $t \in G_I$, $1_{H_u} = \sigma(p_t) = \sigma(u_t^* u_t)$. Since $u$ is a partial *-representation, it follows that
\[
\sigma(u_s)\sigma(u_t) = \sigma(u_s u_t u_t^* u_t) = \sigma(u_{st} u_t^* u_t) = \sigma(u_{st}).
\]
Hence $\sigma \circ u$ is a unitary representation of $G_I$.

The pair $(\sigma,\sigma \circ u)$ is a covariant representation for the C*-dynamical system $(I(A),H)$. Since $H$ is amenable, we obtain a representation $\pi : I(A) \times H \to \B(H_u)$. Let $\phi_{(I,H)} : I(A) \times_r G \to \B(H_u)$ denote the composition of the conditional expectation from $I(A) \times_r G$ to $I(A) \times H$ with $\pi$. Then for $b \in I(A)$ and $s \in H$, $\phi_{(I,H)}(b \lambda_s) = b \chi_H(s) u_s$. The result now follows from the fact that $\phi_{(I,H)}|_A = \sigma|_A$ and $\ker \sigma = J$. 
\end{proof}

\begin{thm} \label{thm:main}
Let $(A,G)$ be a C*-dynamical system. If every amenable $(A,G)$-covering subset of $\Glimm_p(A) \times \Sub(G)$ contains a trivial $(A,G)$-covering subset, then $(A,G)$ has the intersection property. If $(A,G)$ has vanishing obstruction, then the converse is also true.
\end{thm}

\begin{proof}
($\Rightarrow$)
Suppose $(A,G)$ does not have the intersection property. Then Theorem \ref{thm:vanishing-obstruction-intersection-property-iff-properly-outer} implies that its injective hull $(I_G(A),G)$ is not properly outer. Hence by Proposition \ref{prop:continuous-dense-g-delta}, there is a dense $G_\delta$ subset $L \subseteq \Glimm(I_G(A))$ with the property that $G_z \leq G_{A \cap z}$ for $z \in L$ and the stabilizer map 
\[
\Glimm(I_G(A)) \to \Sub(G) : z \to G_z
\]
is continuous on $L$. Furthermore, since $(I_G(A),G)$ is not properly outer, the set $\{z \in L : G_z \ne \{e\} \}$ is non-empty.

Let $Z \subseteq \Glimm(I_G(A)) \times \Sub(G)$ denote the closure of the set
\[
\{ (z, G_z) : z \in L \}.
\]
Since $\Glimm(I_G(A)) \times \Sub(G)$ is compact, $Z$ is compact. Hence since $L$ is dense in $\Glimm(I_G(A))$, $\pr_1(Z) = \Glimm(I_G(A))$.

First we claim that for $(z,H) \in Z$, $H \leq G_z$. To see this, let $\{z_i\}$ be a net in $L$ such that $\lim (z_i,G_z) = (z,H)$ for $z \in \Glimm(I_G(A))$ and $H \leq G$. For $s \in H$, the definition of convergence in the Chabauty topology implies that $s$ is eventually in $G_{z_i}$, so $z_i$ is eventually in $\supp(q_s)$. Since $\supp(q_s)$ is clopen, $z \in \supp(q_s)$. Hence $s \in G_z$ and $H \leq G_z$.

Next, we claim that for $z \in L$,
\[
\{H \in \Sub(G) : (z,H) \in Z \} = \{ G_z \}.
\]
To see this, observe that if $\{ z_i \}$ is a net in $L$ such that $\lim (z_i, G_{z_i}) = (z, H)$ for $H \leq G$, then it follows from the continuity of the stabilizer map on $L$ that $H = G_z$.

By Proposition \ref{prop:intersections-glimm-ideals-are-pseudo-glimm}, $A \cap z$ is a pseudo-Glimm ideal for every $z \in \Glimm(I_G(A))$. Define $X \subseteq \Glimm_p(A) \times \Sub_a(G)$ by
\[
X = \{ (A \cap z, H) : (z,H) \in Z \}.
\]
Since the intersection map from $\Glimm(I_G(A))$ to $\Glimm_p(A)$ is continuous in the strong topology and $Z$ is compact, $X$ is closed.

From above, for $(z,H) \in Z$, $H \leq G_z \leq G_{A \cap z}$, so Proposition \ref{prop:quasi-stabilizer-for-injective-is-amenable} implies that $H$ is amenable. Furthermore, since $\pr_1(Z) = \Glimm(I_G(A))$,
\[
\pr_1(X) = \{ A \cap z : z \in \Glimm(I_G(A)) \}.
\]
Therefore, Proposition \ref{prop:inclusion-pseudo-glimm-ideals} implies that $X$ is an amenable $(A,G)$-covering subset.

We claim $X$ does not contain a trivial $(A,G)$-covering subset. To see this, suppose for the sake of contradiction that $X' \subseteq X$ is a trivial covering subset of $\Glimm_p(A) \times \Sub(G)$. Then there is an invariant closed subset $F \subseteq \Glimm(I_G(A))$ such that $\cap_{z \in F} A \cap z = 0$ and
\[
X' = \{ (A \cap z, \{e\}) : z \in F \}.
\]
Proposition \ref{prop:quasi-glimm-ideals-covering} implies that $F = \Glimm(I_G(A))$. Hence
\[
\{ (z, \{e\}) : z \in \Glimm(I_G(A)) \} \subseteq Z.
\]
From above, $G_z = \{e\}$ for every $z \in L$, giving a contradiction.

($\Leftarrow$)
Suppose $(A,G)$ has vanishing obstruction and let
\[
X \subseteq \Glimm_p(A) \times \Sub(G)
\]
be an amenable $(A,G)$-covering subset that does not contain a trivial covering subset. Then there is a primitive ideal $P_0 \in \Prim(A)$ such that whenever $(I,H) \in X$ satisfies $P_0 \supseteq I$, then $H \ne \{e\}$. This implies there is a neighborhood $U_I \subseteq \Glimm_p(A)$ of $I$ and a neighborhood $V_I \subseteq \Sub(G)$ of $\{e\}$ such that whenever $(I',H') \in X$ satisfies $I' \in U_I$, then $H' \notin V_I$. By the definition of the Chabauty topology, there is a finite subset $E_I \subseteq G \setminus \{e\}$ such that $V_I = \{H \in \Sub(G) : H \cap E_I = \emptyset \}$.

The set $\{I \in \Glimm_p(A) : P_0 \supseteq I\}$ is closed and hence compact. From above it follows that there is a finite subset $E \subseteq G \setminus \{e\}$ such that whenever $\{(I_i,H_i)\}$ is a net in $X$ with $\lim (I_i,H_i) = (I,H)$ and $P_0 \supseteq I$, then eventually $H_i \cap E \ne \emptyset$.

For $(I,H) \in X$, let $\phi_{(I,H)} : I(A) \times_r G \to \B(H_{(I,H)})$ be defined as in Lemma \ref{lem:intermediate-map}. Define $d \in I(A) \times_r G$ by
\[
d = \sum_{s \in E} u_s^* \lambda_s.
\]
Then from above, whenever $\{(I_i,H_i)\}$ is a net in $X$ with $\lim(I_i,H_i) = (I,H) \in X$ and $P_0 \supseteq I$, eventually
\[
\phi_{(I_i,H_i)}(d) = \sum_{s \in E} \chi_{H_i}(s) = |E \cap H_i| \geq 1.
\]

Let $\bR = (R_n) \subseteq MS(I(A) \times_r G)$ denote the set of all possible matrix states obtained by compressing the maps in $\{\phi_{(I,H)} : (I,H) \in X \}$ to finite dimensional subspaces and conjugating by unitaries. Then $\bR$ is invariant. Let $\bK = (K_n) \subseteq MS(I(A) \times_r G)$ denote the closed matrix convex hull of $\mathbf{R}$. Then $\bK$ is also invariant, so $(\bK,G)$ is a matrix affine dynamical system.

By \cite{WW1999}*{Example 2.3} and \cite{A1969}*{Corollary 1.4.3}, the matrix extreme points of the matrix state space $MS(A)$ of $A$ are precisely the compressions of irreducible representations of $A$ to finite dimensional subspaces. Since $X$ is covering, it follows from the construction of the maps $\phi_{(I,H)}$ for $(I,H) \in X$ that restricting to $A$ maps $\bR$ onto a subset of $MS(A)$ that contains all the matrix extreme points of $MS(A)$. Hence by \cite{WW1999}*{Theorem 4.3}, restricting to $A$ maps $\bK$ onto $MS(A)$. It follows from Corollary \ref{cor:contains-boundary} that $(\bK,G)$ contains a boundary, say $(\bK',G)$ of $(MS(A),G)$. Write $\bK' = (K'_n)$. 

Let $\alpha$ be a pure state on $A$ with GNS representation $\pi_\alpha : A \to \B(H)$ satisfying $\ker \pi_\alpha = P_0$. Suppose $\mu \in R_1$ satisfies $\mu|_A = \alpha$. Then from above there is a net $(\mu_i)$ in $R_1$ and a net $((I_i,H_I))$ in $X$ such that $\mu_i$ is a compression of a unitary conjugate of $\phi_{(I_i,H_i)}$ and $\lim \mu_i = \mu$. By passing to a subnet we can assume $\lim (I_i,H_i) = (I,H) \in X$. Then for $a \in I$ and $b,c \in A$, $\lim \|\phi_{(I_i,H_i)}(bac)\| = \|bac + I_i\| = 0$. Hence
\[
|\alpha(bac)| = \lim |\alpha_i(bac)| \leq \lim \|\phi_{(I_i,H_i)}(bac)\| = 0,
\]
implying $a \in \ker \pi_\alpha = P_0$. Therefore, $P_0 \supseteq I$. From above, eventually $\phi_{(I_i,H_i)}(d) \geq 1$. Hence $\alpha(d) = \lim \alpha_i(d) \geq 1$.

By \cite{WW1999}*{Theorem 4.6}, every $\mu \in K_1$ is a limit of convex combinations of states in the closure of $R_1$. If $\mu|_A = \alpha$, then since $\alpha$ is pure, $\mu$ is a limit of convex combinations of states in the closure of $R_1$ that restrict to $\alpha$. Hence from above, $\mu(d) \geq 1$. It follows that the boundary $\bK'$ is non-canonical. 

By Theorem \ref{thm:nc-boundary-correspondence-pseudo-expectation}, there is a non-canonical pseudo-expectation for $(A,G)$, so it follows from Theorem \ref{thm:properly-outer-implies-unique-pseudo-expectation} that $(I_G(A),G)$ is not properly outer. By Theorem \ref{thm:vanishing-obstruction-intersection-property-iff-properly-outer}, we conclude that $(A,G)$ does not have the intersection property.
\end{proof}

\begin{cor}
Let $(A,G)$ be a C*-dynamical system such that $A$ is prime and let $G_A = \{s \in G : \alpha_s \text{ is quasi-inner} \}$. If every amenable $(A,G)$-covering subset of $\Glimm_p(A) \times \Sub(G_A)$ contains a trivial $(A,G)$-covering subset, then $(A,G)$ has the intersection property. If $(A,G)$ has vanishing Mackey obstruction $H^2(G_A,\mathbb{T})$, then the converse is also true.
\end{cor}

\begin{proof}
By assumption, for every pseudo-Glimm ideal $I \in \Glimm_p(A)$, $G_I = G_A$. Since $A$ is prime, it follows as in Example \ref{ex:partial-cohomology-2} that $(A,G)$ has vanishing obstruction if and only if the Mackey obstruction $H^2(G_A,\mathbb{T})$ vanishes. Hence the result follows from Theorem \ref{thm:main}. 
\end{proof}

\subsection{Intrinsic characterization in the minimal case}

The analysis required to apply Theorem \ref{thm:main} is greatly simplified when $(A,G)$ is minimal, i.e. when $A$ has no non-trivial invariant ideals. In this case, $(A,G)$ has the intersection property if and only if the reduced crossed product $A \times_r G$ is simple.

\begin{thm} \label{thm:main-minimal}
Let $(A,G)$ be a minimal C*-dynamical system. Suppose that for every pseudo-Glimm ideal $I \in \Glimm_p(A)$ and every amenable subgroup $H \leq G_I$, there is a net $(s_i)$ in $G$ such that $\lim s_i H s_i^{-1} = \{e\}$ in the Chabauty topology. Then the reduced crossed product $A \times_r G$ is simple. If $(A,G)$ has vanishing obstruction, then the converse is also true.
\end{thm}

\begin{proof}
($\Rightarrow$)
Suppose that for every $I \in \Glimm_p(A)$ and amenable subgroup $H \leq G_I$, there is a net $(s_i)$ in $G$ such that $\lim s_i H s_i^{-1} = \{e\}$. Let $X \subseteq \Glimm_p(A) \times \Sub(G)$ be an amenable $(A,G)$-covering subset. Fix $(I,H) \in X$. By assumption there is a net $(s_i)$ in $G$ such that $\lim s_i H s_i^{-1} = \{e\}$. By compactness, the net $(s_n I, s_n H s_n^{-1})$ has a convergent subnet. Thus there is $(I',\{e\}) \in X$. Since $I'$ is proper, it follows from the minimality of $(A,G)$ that the closure of the subset $\{(sI',\{e\}) : s \in G\} \subseteq X$ is covering. Hence every amenable covering subset of $\Glimm_p(A) \times \Sub(G)$ contains a trivial $(A,G)$-covering subset. By Theorem \ref{thm:main}, $(A,G)$ has the intersection property.

($\Leftarrow$)
Suppose $(A,G)$ has vanishing obstruction and the intersection property. Then it follows from Theorem \ref{thm:vanishing-obstruction-intersection-property-iff-properly-outer} that $(I(A),G)$ also has the intersection property. For $I \in \Glimm_p(A)$ and an amenable subgroup $H \leq G_I$, let $J$ be a proper primal ideal in $I(A)$ such that $I = A \cap J$. Let $Y \subseteq \Glimm_p(I(A)) \times \Sub(G)$ denote the closure of the orbit of $(J,H)$. Since $s \in G_I$, $s \in G_J$, so for $(J',H') \in Y$, an argument similar to the proof of Proposition \ref{prop:continuous-dense-g-delta} implies that $H' \leq G_{J'}$. Furthermore, since $(A,G)$ is minimal, \cite{H1985}*{Proposition 6.4} implies that $(I(A),G)$ is also minimal. Hence $Y$ is an amenable $(I(A),G)$-covering subset. Theorem \ref{thm:main} now implies that $Y$ contains a trivial $(I(A),G)$-covering subset. It follows that there is a net $(s_i)$ in $G$ such that $\lim s_i H s_i^{-1} = \{e\}$.
\end{proof}

\begin{cor}
Let $(A,G)$ be a minimal C*-dynamical system such that $A$ is prime and let $G_A = \{s \in G : \alpha_s \text{ is quasi-inner} \}$. If for every amenable subgroup $H \leq G_A$, there is a net $(s_i)$ in $G$ such that $\lim s_i H s_i^{-1} = \{e\}$ in the Chabauty topology, then the reduced crossed product $A \times_r G$ is simple. If $G_A$ has vanishing Mackey obstruction $H^2(G_A,\mathbb{T})$, then the converse is also true.
\end{cor}

\begin{proof}
Since $A$ is prime, it follows as in Example \ref{ex:partial-cohomology-2} that $(A,G)$ has vanishing obstruction if and only if the Mackey obstruction $H^2(G_A,\mathbb{T})$ vanishes. Hence the result follows from Theorem \ref{thm:main-minimal}. 
\end{proof}


\end{document}